\documentclass[reqno,oneside]{amsart}

\usepackage{amssymb,amsmath,amssymb,amsfonts,amsthm}
\usepackage{bm}
\usepackage{cases}
\usepackage{cite}
\usepackage{float}
\usepackage{geometry}
\usepackage{hyperref}
\usepackage{mathtools}
\usepackage[subrefformat=parens]{subcaption}
\usepackage{tikz}
\usepackage{url}

\numberwithin{equation}{section}

\theoremstyle{plain}
\newtheorem{theorem}{Theorem}[section]
\newtheorem{lemma}{Lemma}[section]

\newtheorem{corollary}{Corollary}[section]

\newtheorem*{conjecture*}{Conjecture}

\theoremstyle{definition}

\newtheorem{example}{Example}[section]

\allowdisplaybreaks[4]

\def\BC{\mathbb C}
\def\BR{\mathbb R}

\def\bx{\bm{x}}
\def\cA{\mathcal A}
\def\rA{\mathrm A}
\def\rB{\mathrm B}
\def\rC{\mathrm C}
\def\rIm{\mathrm{Im}}

\def\rR{\mathrm R}
\def\rS{\mathrm S}
\def\T{\mathrm T}
\def\rd{\mathrm d}
\def\diag{\mathrm{diag}}
\def\rdiv{\mathrm{div}}
\def\e{\mathrm e}
\def\ri{\mathrm i}
\def\Ga{\Gamma}
\def\De{\Delta}
\def\Om{\Omega}
\def\al{\alpha}
\def\be{\beta}
\def\ga{\gamma}
\def\de{\delta}
\def\ve{\varepsilon}
\def\te{\theta}
\def\ka{\kappa}
\def\la{\lambda}
\def\si{\sigma}
\def\vp{\varphi}
\def\f{\frac}
\def\nb{\nabla}
\def\ov{\overline}
\def\pa{\partial}
\def\wh{\widehat}
\def\un{\underline}
\DeclareMathOperator*{\res}{Res}

\title[Sharp decay estimates for coupled subdiffusion systems]{Sharp decay estimates and numerical analysis for weakly coupled systems of two subdiffusion equations}

\author[Z. Li]{Zhiyuan Li$^1$}
\thanks{$^1$ School of Mathematics and Statistics, Ningbo University, 818 Fenghua Road, Ningbo 315211, China. E-mail: {\tt lizhiyuan@nbu.edu.cn}}

\author[Y. Liu]{Yikan Liu$^2$}
\thanks{$^2$ Corresponding author. Department of Mathematics, Kyoto University, Kitashirakawa-Oiwakecho, Sakyo-ku, Kyoto 606-8502, Japan. E-mail: {\tt liu.yikan.8z@kyoto-u.ac.jp}}

\author[K. Wada]{Kazuma Wada$^3$}
\thanks{$^3$ NS Solutions Corporation, Toranomon Hills Business Tower, 1-17-1 Toranomon, Minato-ku, Tokyo 105-6417, Japan. E-mail: {\tt wada.kazuma.a4a@jp.nssol.nipponsteel.com}}


\begin{document}

\begin{abstract}
This paper investigates the initial-boundary value problem for weakly coupled systems of time-fractional subdiffusion equations with spatially and temporally varying coupling coefficients. By combining the energy method with the coercivity of fractional derivatives, we convert the original partial differential equations into a coupled ordinary differential system. Through Laplace transform and maximum principle arguments, we reveal a dichotomy in decay behavior: When the highest fractional order is less than one, solutions exhibit sublinear decay, whereas systems with the highest order equal to one demonstrate a distinct superlinear decay pattern. This phenomenon fundamentally distinguishes coupled systems from single fractional diffusion equations, where such accelerated superlinear decay never occurs. Numerical experiments employing finite difference methods and implicit discretization schemes validate the theoretical findings.
\end{abstract}

\maketitle	

\section{Introduction and the main result}

Let $\al,\be\in\BR_+:=(0,+\infty)$ be constants satisfying $1\ge\al>\be$ and $\Om\subset\BR^d$ ($d=1,2,\dots$) be a bounded domain with a sufficiently smooth boundary $\pa\Om$. In this article, we investigate the following initial-boundary value problem for a coupled subdiffusion system of two equations
\begin{equation}\label{eq-gov}
\begin{cases}
\begin{aligned}
& \pa_t^\al(u-u_0)-\rdiv(\bm A(\bm x)\nb u)+c_{11}(\bm x,t)u+c_{12}(\bm x,t)v=0,\\
& \pa_t^\be(v-v_0)-\rdiv(\bm B(\bm x)\nb v)+c_{21}(\bm x,t)u+c_{22}(\bm x,t)v=0
\end{aligned}
& \mbox{in }\Om\times\BR_+,\\
u=v=0 & \mbox{on }\pa\Om\times\BR_+.
\end{cases}
\end{equation}
Here $\pa_t^\al$ represents the $\al$-th order derivative in time which is defined as the inverse of the $\al$-th order Riemann-Liouville integral operator
\[
J^\al:L^2(0,T)\longrightarrow L^2(0,T),\quad J^\al f(t):=\int_0^t\f{\tau^{\al-1}}{\Ga(\al)}f(t-\tau)\,\rd\tau
\]
for any $T>0$, where $\Ga(\,\cdot\,)$ is the Gamma function. It was revealed in \cite{GL15,KR20} that the domain of $\pa_t^\al$, written as
\[
D(\pa_t^\al)=J^\al(L^2(0,T))=H_\al(0,T),
\]
belongs to some fractional Sobolev space which only allows a pointwise definition for $\al>1/2$. Then the notation $\pa_t^\al(u-u_0)$ means $u(\bm x,\,\cdot\,)-u_0(\bm x)\in H_\al(0,T)$ for a.e.\! $\bm x\in\Om$ and any $T>0$, and the initial condition $u=u_0$ in $\Om\times\{0\}$ only makes usual sense for $\al>1/2$. Especially, for $\al=1$ one can interpret $\pa_t^\al(u-u_0)=\pa_t u$ as the usual partial derivative in time. In the spatial direction of \eqref{eq-gov}, let the principal coefficients $\bm A,\bm B$ in the elliptic parts and the zeroth order coefficients $c_{k\ell}$ ($k,\ell=1,2$) satisfy
\begin{gather}
\bm A,\bm B\in C^1(\ov\Om;\BR_{\mathrm{sym}}^{d\times d}),\quad\bm A(\bm x)\bm\xi\cdot\bm\xi\ge\ka_0|\bm\xi|^2,\ \bm B(\bm x)\bm\xi\cdot\bm\xi\ge\ka_0|\bm\xi|^2,\quad\forall\,\bm x\in\ov\Om,\ \forall\,\bm\xi\in\BR^d,\label{eq-assume1}\\
c_{k\ell}\in L^\infty(\Om\times\BR_+),\quad c_{11}\ge0,\ c_{22}\ge0\quad\mbox{in }\Om\times\BR_+,\label{eq-assume2}
\end{gather}
where $\ka_0>0$ is a constant. Throughout, we abbreviate, for example, $u(t)=u(\bm x,t)$ by considering $u$ as a vector-valued function from $\BR_+$ to some suitable function space in $\Om$.

The physical background of fractional diffusion equations comes from the extension of classical diffusion models. The introduction of fractional derivatives effectively captures anomalous phenomenon such as non-Gaussian profiles and long-term memory effects in porous media and biological tissues, which cannot be well described by traditional integer-order diffusion equations, see e.g., \cite{Adam1992,Hatano1998,Metzler2000} and the references therein.  Fractional diffusion systems employ nonlocal operators and coupling mechanisms to not only describe fundamental diffusion processes such as single equations but also incorporate complex factors, which have significant applications in fields such as chemical reactions, biological mass transfer, and environmental engineering. For example, in the transport of solute through porous media, they can integrate multiple influencing factors into a unified model and, therefore, can describe the dynamic mass exchange process between mobile and immobile zones  \cite{Doerries2022,LiWen2020,Schumer2003,SunNiu2021} and the references therein.

The long-time asymptotic behavior is the most remarkable difference between fractional and non-fractional equations.  This fundamental difference was first quantitatively characterized through single time-fractional equation studies: Sakamoto and Yamamoto \cite{SY11} demonstrated that solutions to time-fractional diffusion equations exhibit fractional polynomial decay, where the memory effect fundamentally alters the exponential decay pattern of classical diffusion. Subsequent extensions to multi-term fractional models in \cite{LLY15} revealed that the lowest order dominates the long-term decay. Remarkably, Luchko et al. \cite{Luchko2014} and Kubica and Ryszewska \cite{Kubica} observed logarithmic-type decay in distributed-order models, which generalize the multi-term time-fractional diffusion equation (featuring a finite sum of fractional derivatives) to a framework involving continuous integration over infinitely many fractional derivatives of distinct orders. Vergara and Zacher \cite{Zacher2015} established the quantitative relation between decay rates and fractional orders in single equations with time-dependent coefficients. Li, Huang, Liu \cite{LHL23} studied the same topic for a weakly coupled subdiffusion systems of $K$ components
\begin{equation}\label{eq-couple}
\pa_t^{\al_k}(u_k-u_0^{(k)})-\rdiv(\bm A_k(\bm x)\nb u_k)+\sum_{\ell=1}^K c_{k\ell}(\bm x)u_\ell=0\quad\mbox{in }\Om\times\BR_+,\ k=1,\dots,K,
\end{equation}
where $1>\al_1\ge\cdots\ge\al_K>0$, $\bm A_k$ satisfy the same assumption as \eqref{eq-assume1} and $c_{k\ell}\in L^\infty(\Om)$ depend only on $\bm x$ ($k,\ell=1,\dots,K$). Under a certain non-positivity assumption on $c_{k\ell}$, \cite[Theorem 2]{LHL23} established the decay estimate (see also Lemma \ref{lem-wp})
\begin{equation}\label{eq-asymp0}
\sum_{k=1}^K\|u_k(t)\|_{H^2(\Om)}\le C\sum_{k=1}^K\|u_0^{(k)}\|_{L^2(\Om)}\,t^{-\al_K},\quad\forall\,t\gg1.
\end{equation}
In other words, the decay of the solution to \eqref{eq-couple} is dominated by the lowest order $\al_K$ of fractional derivatives in time. Moreover, it was asserted that the decay rate $t^{-\al_K}$ is sharp if the initial value $u_0^{(K)}$ of the $K$-th component does not vanish identically in $\Om$. However, the sharp decay estimate for \eqref{eq-couple} in the case of $u_0^{(K)}\equiv0$ in $\Om$ was not studied in the literature, which is a natural problem but still remains open.

Keeping the above background in mind, in this paper we are interested in sharpening the decay estimate \eqref{eq-asymp0} under the aforementioned situation. As the first attempt, we restrict ourselves to the formulation \eqref{eq-gov} of two components. Since $v_0\equiv0$ in $\Om$, the time evolution of $v$ relies completely on the supply from $u$. Therefore, one can imagine that the higher order $\al$ may appear to influence the long-time behavior of the solution. To verify our conjecture, we perform numerical tests under various settings, including both cases of $\al<1$ and $\al=1$. As is reported in Section \ref{sec-test}, we observed a decay rate of $t^{-\al}$ for $\al<1$ as predicted. Surprisingly, we witnessed an unexpected decay pattern for $\al=1$ which seems to be never discovered in the literature. These observations motivate us to establish the following main result of this manuscript.
\begin{theorem}\label{thm-asymp}
Let $v_0\equiv0$ in $\Om,$ $u_0\in L^2(\Om)$ and assumptions {\rm\eqref{eq-assume1}--\eqref{eq-assume2}} be satisfied. Let $C_\Om>0$ be the optimal constant in the Poincar\'e inequality depending only on $\Om,$ i.e.,
\begin{equation}\label{eq-Poincare}
\|\psi\|_{L^2(\Om)}\le C_\Om\|\nb\psi\|_{L^2(\Om)},\quad\forall\,\psi\in H_0^1(\Om),
\end{equation}
and assume that
\begin{equation}\label{eq-assume0}
\f{\ka_0}{C_\Om^2}>\ka_1:=\max\left\{\|c_{12}\|_{L^\infty(\Om\times\BR_+)},\|c_{21}\|_{L^\infty(\Om\times\BR_+)}\right\}.
\end{equation}
Then there exists a constant $C>0$ such that the solution $(u,v)$ to \eqref{eq-gov} satisfies
\begin{equation}\label{eq-asymp1}
\|u(t)\|_{L^2(\Om)}+\|v(t)\|_{L^2(\Om)}\le\begin{cases}
C\|u_0\|_{L^2(\Om)}\,t^{-\al}, & \al<1,\\
C\|u_0\|_{L^2(\Om)}\,t^{-(1+\be)}, & \al=1,
\end{cases}
\quad\forall\,t>1.
\end{equation}
\end{theorem}

The remainder of the paper is organized as follows. In Section \ref{sec-pre}, several necessary results are given, including asymptotic estimates of Mittag-Leffler functions, the coercivity of fractional derivatives, and the residue theory on the complex plane cutting off the negative real axis. In Section \ref{sec-proof}, we establish the maximum principle and the asymptotic behavior of the solution to a coupled fractional ordinary system, from which we further employ the coercivity to give a proof of the main theorem. In Section \ref{sec-scheme}, we propose semi-implicit and fully implicit numerical schemes for fractional diffusion equations, and perform stability analysis to ensure the reliability of numerical simulations. Finally, in Section \ref{sec-test}, we verify the validity of the theoretical results through specific numerical experiments demonstrating the consistency between the numerical results and the theoretical analysis.

\section{Preliminaries}\label{sec-pre}

In the sequel, by $\|\cdot\|$ and $(\,\cdot\,,\,\cdot\,)$ we denote the norm and the inner product of $L^2(\Om)$ respectively. In addition, by $C>0$ we denote generic constants which may change from line to line.

We first recall the frequently used Mittag-Leffler functions $E_{\eta,\mu}(z)$ with two parameters $\eta>0$, $\mu\in\BR$ defined by (e.g.\! Podlubny \cite{PF98})
\[
E_{\eta,\mu}(z)=\sum_{k=0}^\infty\f{z^k}{\Ga(\eta k+\mu)},\quad z\in\BC.
\]
Many properties of the Mittag-Leffler functions play important roles in fractional differential equations, and here we only invoke their estimates and positivity for our use in this article.

\begin{lemma}\label{lem-ML}
\begin{enumerate}
\item Let $0<\eta<2$ and $\mu\in\BR$. Then there exists a constant $C>0$ depending only on $\eta,\mu$ such that
$$
|E_{\eta,\mu}(-z)|\le\f C{1+z},\quad\forall\,z\ge0.
$$
\item If $0<\eta\le1,$ then
\[
0<E_{\eta,1}(-z)\le1,\quad0<E_{\eta,\eta}(-z)\le1,\quad\forall\,z\ge0.
\]
\end{enumerate}
\end{lemma}

Further details concerning the above lemma can be found in Podlubny \cite{PF98}.

Next, let us turn to the coercivity of the fractional derivative $\pa_t^\al$. For the usual first order partial derivative $\pa_t$, it is readily seen for a smooth function $w$ defined in $\Om\times\BR_+$ that
\[
(w(t),\pa_t w(t))=\f12\f\rd{\rd t}\left(\|w(t)\|^2\right)=\|w(t)\|\,\pa_t\|w(t)\|.
\]
The following lemma generalizes the above fact to $\pa_t^\al$ with a slight modification.

\begin{lemma}\label{lem-coer}
For any $T>0,$ let $w\in L^2(0,T;L^2(\Om))$ satisfy $w-w_0\in H_\al(0,T;L^2(\Om))$ with $w_0\in L^2(\Om)$. Then there holds
\begin{equation}\label{esti-coer}
(w(t),\pa_t^\al(w(t)-w_0))\ge\|w(t)\|\,\pa_t^\al(\|w(t)\|-\|w_0\|)
\end{equation}
for a.e.\! $t\in(0,T)$.
\end{lemma}

\begin{proof}
The case of $\al=1$ is trivial and we only deal with that of $\al<1$. 

The above inequality can be proved by a density argument. To this end, we first assume $w\in C^1([0,T];L^2(\Om))$. Then it follows from Li, Huang and Yamamoto \cite[Lemma 4.1]{LHY21} that
\begin{equation}\label{esti-coercivity}
(w(t),\rd_t^\al w(t))\ge\|w(t)\|\,\rd_t^\al\|w(t)\|,\quad0<t<T,
\end{equation}
where $\rd_t^\al$ stands for the $\al$-th order Caputo derivative defined by
\begin{equation}\label{eq-Caputo}
\rd_t^\al f(t):=\int_0^t\f{\tau^{-\al}}{\Ga(1-\al)}f'(t-\tau)\,\rd\tau,\quad f\in C^1[0,T].
\end{equation}
Since it is known that $\rd_t^\al$ and $\pa_t^\al$ satisfy
\[
\rd_t^\al f(t)=\pa_t^\al(f(t)-f(0)),\quad f\in C^1[0,T],
\]
one can rewrite \eqref{esti-coercivity} in terms of $\pa_t^\al$, so that \eqref{esti-coer} holds for any $t\in(0,T)$ and $w\in C^1([0,T];L^2(\Om))$ or equivalently $w-w_0\in{}_0C^1([0,T];L^2(\Om))$, where $_0C^1[0,T]:=\{f\in C^1[0,T]\mid f(0)=0\}$.

Next, we treat the case of $w-w_0\in H_\al(0,T;L^2(\Om))$. According to \cite[Lemma 2.2]{KR20},  the space $_0C^1[0,T]$ is dense in $H_\al(0,T)$. Therefore, we can choose a sequence $\{w_m\}_{m=1}^\infty\subset C^1([0,T];L^2(\Om))$ such that
\[
w_m-w_0\in{}_0C^1([0,T];L^2(\Om)),\quad w_m-w_0\longrightarrow w-w_0\mbox{ in }H_\al(0,T;L^2(\Om))\mbox{ as }m\to\infty.
\]
Therefore, we have
\[
\pa_t^\al(w_m-w_0)\longrightarrow\pa_t^\al(w-w_0)\quad\mbox{in }L^2(0,T;L^2(\Om))\mbox{ as }m\to\infty
\]
and for each $w_n$, there holds
\[
(w_m(t),\pa_t^\al(w_m(t)-w_0))\ge\|w_m(t)\|\,\pa_t^\al(\|w_m(t)\|-\|w_0\|),\quad0<t<T.
\]
Then the coercivity \eqref{esti-coer} can be derived by simply passing $m\to\infty$ in the above inequality.
\end{proof}

The proof of our main result relies heavily on the application of the residue theorem, and we first recall the following classical result regarding holomorphic functions on the complex plane.

\begin{lemma}\label{lem-residue0}
Let $D$ be an open domain in the complex plane $\BC$. Then there holds
$$
\oint_\ga g(z)\,\rd z=2\pi\,\ri\sum_{k=1}^N\res_{z=z_k}g(z),
$$
where $f$ is a holomorphic function defined in $D\setminus\{z_1,\dots,z_N\}$ and $\ga$ is a rectifiable and closed Jordan curve in $D$ enclosing but not passing through $z_1,\dots,z_N$.
\end{lemma}

Now for a multi-valued meromorphic function $g$ satisfying $|g(z)|\le C|z|^{-p}$ with some constant $p>0$, suppose that all poles $\{z_1,\dots,z_N\}$ of $g(z)$ are distributed in the left half complex plane excluding the real axis. On the basis of Lemma \ref{lem-residue0}, we establish a more general residue result which is suitable also for the function $g(z)$ on the complex plane cutting off the negative axis.

\begin{lemma}\label{lem-residue1}
Let $g$ be a meromorphic function on the complex plane cutting off the negative axis. Assume that there exists a constant $p>0$ such that
\begin{equation}\label{eq-decay-p}
|g(z)|\le C|z|^{-p},\quad z\in\BC
\end{equation}
and all poles $z_1,\dots,z_N$ of $g(z)$ have negative real parts and non-vanishing imaginary parts. Then there holds
$$
\int_{s_0-\ri\,\infty}^{s_0+\ri\,\infty}g(z)\,\rd z =\int_\ga g(z)\,\rd z+2\pi\,\ri\sum_{k=1}^N\res_{z=z_k}g(z),
$$
where $s_0>0$ is a constant and the integral path $\ga$ is defined as
\[
\ga:=\begin{cases}
r\,\e^{\ri\,\pi}, & r>0,\\
r\,\e^{-\ri\,\pi}, & r>0.
\end{cases}
\]
\end{lemma}

\begin{proof}
Let us take an integral path $\ga_{R,\te_0,\ve}$ as the curve illustrated in Figure \ref{fig-curve}.
\begin{figure}[htbp]\centering
\begin{tikzpicture}
\draw[line width=1pt,->] (-5,0) --(5,0) node[right]{Re};
\draw[line width=1pt,->] (0,-4) --(0,4) node[left]{Im};
\draw[blue,line width=1pt] (-0.7,-0.7) arc (-135:135:1);
\draw[line width=1pt] (-0.7,-0.7)-- (-3,-3);
\draw[line width=1pt] (-0.7,0.7)-- (-3,3);
\draw[line width=1pt] (-3,3)-- (3,3);
\draw[line width=1pt] (-3,-3)-- (3,-3);
\draw[line width=1pt,->] (3,-3)-- (3,2);
\draw[line width=1pt] (3,2)-- (3,3);
\draw[line width=1pt,->] (4,2) --(5,2) node[above] {$\ga_{R,\te_0,\ve}$};
\fill[blue](0,-3)circle(2pt)node[below]{$-\ri\,R$\qquad\quad\quad};
\fill[blue](3,0)circle(2pt)node[below]{\qquad$s_0$};
\fill[blue](0,3)circle(2pt)node[below]{$\ri\,R$\qquad\quad};
\draw[gray,line width=4pt,opacity=0.5] (-5,0) --(0,0);
\end{tikzpicture}
\caption{The choice of an integral path $\ga_{R,\te_0,\ve}$.}\label{fig-curve}
\end{figure}
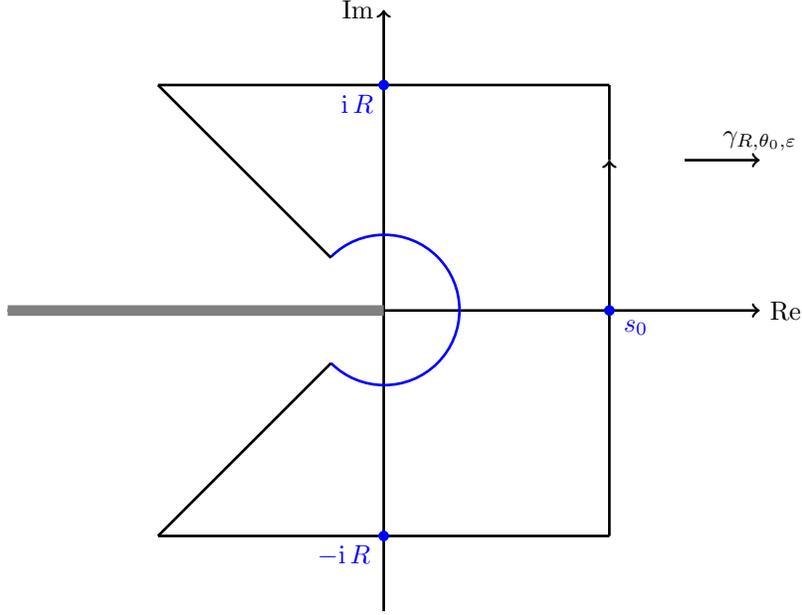
Here $\ga_{R,\te_0,\ve}$ is required not to pass through $z_1,\dots,z_N$, and $R>0$ is taken sufficiently large. Then Lemma \ref{lem-residue0} indicates
$$
\oint_\ga g(z)\,\rd z=2\pi\,\ri\sum_{k=1}^N\res_{z=z_k}g(z).
$$
In addition, applying the assumption \eqref{eq-decay-p} for $|z|>1$, we pass $R\to+\infty$ to derive
$$
\int_{\Ga_{\ve,\te_0}}g(z)\,\rd z+\int_{s_0-\ri\,\infty}^{s_0+\ri\,\infty}g(z)\,\rd z=2\pi\,\ri\sum_{k=1}^N\res_{z=z_k}g(z),
$$
where
\[
\Ga_{\ve,\te_0}:=\begin{cases}
r\,\e^{\ri\,\te_0}, & r>\ve,\\
\ve\,\e^{\ri\,\te}, & -\te_0<\te<\te_0,\\
r\,\e^{-\ri\,\te_0}, & r>\ve.
\end{cases}
\]
At the same time, passing $\te\to\pi$ yields
$$
\f1{2\pi\,\ri}\int_{s_0-\ri\,\infty}^{s_0+\ri\,\infty}g(z)\,\rd z=\sum_{k=1}^N\res_{z=z_k}g(z)+\f1{2\pi\,\ri}\int_{\ga_\ve}g(z)\,\rd z,   
$$
where 
\[
\Ga_\ve:=\begin{cases}
r\,\e^{\ri\,\pi}, & r>\ve,\\
\ve\,\e^{\ri\,\te}, & -\pi<\te<\pi,\\
r\,\e^{-\ri\,\pi}, & r>\ve.
\end{cases}
\]
Finally, passing $\ve\to0$, it arrives at the desired residue theorem on the complex plane cutting off the negative axis, which finishes the proof of Lemma \ref{lem-residue1}.
\end{proof}

Finally, concerning the original problem \eqref{eq-gov}, we revisit the results on the well-posedness and the long-time asymptotic behavior obtained in \cite[Theorems 1--2]{LHL23} and collect necessary parts of them in the following lemma.

\begin{lemma}\label{lem-wp}
Let $1>\al>\be>0,$ $u_0,v_0\in L^2(\Om)$ and assumptions {\rm\eqref{eq-assume1}--\eqref{eq-assume2}} be satisfied.
\begin{enumerate}
\item For any $T>0,$ there exists a unique solution $(u,v)\in(L^2(0,T;H_0^1(\Om)))^2$ such that
\[
\lim_{t\to0}\|u(t)-u_0\|=\lim_{t\to0}\|v(t)-v_0\|=0.
\]
Moreover, there exists a constant $C>0$ depending only on $\Om,\al,\be,\bm A,\bm B,c_{k\ell}$ $(k,\ell=1,2)$ such that
\[
\|u(t)\|_{H^1(\Om)}+\|v(t)\|_{H^1(\Om)}\le C(\|u_0\|+\|v_0\|)\exp(C t)\,t^{-\al/2},\quad\mbox{a.e.\! }t\in(0,T).
\]
\item In addition, assume that $\bm C=(c_{k\ell})_{1\le k,\ell\le2}\in(L^\infty(\Om))^{2\times2}$ is a $t$-independent negative semi-definite matrix-valued function in $\Om$. Then for any fixed $t_0>0,$ there exists a constant $C>0$ depending on $t_0,\Om,\al,\be,\bm A,\bm B,\bm C$ such that
\[
\|u(t)\|_{H^2(\Om)}+\|v(t)\|_{H^2(\Om)}\le C(\|u_0\|+\|v_0\|)\,t^{-\be},\quad\forall\,t\ge t_0.
\]
Moreover, the decay rate $t^{-\be}$ is sharp provided that $v_0\not\equiv0$ in $\Om$.
\end{enumerate}
\end{lemma}

In \cite{LHL23}, all fractional orders of time derivatives were assumed to be strictly smaller than $1$. In this sense, Lemma \ref{lem-wp} fails to cover the special case of $\al=1$ in \eqref{eq-gov}. However, the argument of an iterative construction of the mild solution definitely works for $\al=1$ and thus the well-posedness part in Lemma \ref{lem-wp}(i) still holds.

\section{Proof of the main result}\label{sec-proof}

This section is devoted to the proof of Theorem \ref{thm-asymp}. The main strategy turns out to be a reduction of the original problem \eqref{eq-gov} to a corresponding initial value problem for a coupled system of fractional ordinary differential equations. There are two key ingredients for treating the latter, namely, the maximum principle and the decay estimates of fractional ordinary systems.

\subsection{Maximum principle of fractional ordinary systems}

Let us consider an initial value problem for a coupled fractional ordinary differential system
\begin{equation}\label{eq-gov-3}
\begin{cases}
\pa_t^\al(U-a)+\eta_1U-\mu_1V=F,\\
\pa_t^\be(V-b)-\mu_2U+\eta_2V=G
\end{cases}\mbox{in }\BR_+,
\end{equation}
where $a,b,\eta_1,\eta_2,\mu_1,\mu_2\in\BR$ and $F,G\in L_{\mathrm{loc}}^2(\BR_+)$. Similarly to its partial differential counterpart, one can define the mild solution for \eqref{eq-gov-3} and carry out an iteration argument to show the unique existence of the mild solution in the same manner as that in \cite{LHL23}. We establish the maximum principle of \eqref{eq-gov-3} in the next lemma.

\begin{lemma}\label{lem-max}
Assume $\mu_2>0,$ $a,b,\eta_1,\eta_2,\mu_1,\mu_2\ge0$ and $0\le F,G\in L_{\mathrm{loc}}^2(\BR_+)$. Then there exists a unique solution $(U,V)\in L_{\rm loc}^2(\BR_+)$ to \eqref{eq-gov-3} such that $U\ge0,V\ge0$ in $\BR_+$. Especially, if we further assume $a>0,$ then $U>0,V>0$ in $\BR_+$.
\end{lemma}

\begin{proof}
First we rewrite \eqref{eq-gov-3} as
\[
\begin{cases}
\pa_t^\al(U-a)+\eta_1U=F+\mu_1V,\\
\pa_t^\be(V-b)+\eta_2V=G+\mu_2U
\end{cases}\mbox{in }\BR_+.
\]
Employing the Mittag-Leffler function, it can be easily seen that $(U,V)$ satisfies the integral equation
\begin{equation}\label{eq-inteq}
\left\{\begin{aligned}
U(t) & =U_1(t)+\mu_1\int_0^t\tau^{\al-1}E_{\al,\al}(-\eta_1\tau^\al)V(t-\tau)\,\rd\tau,\\
V(t) & =V_1(t)+\mu_2\int_0^t\tau^{\be-1}E_{\be,\be}(-\eta_2\tau^\be)U(t-\tau)\,\rd\tau,
\end{aligned}\right.
\end{equation}
where
\begin{equation}\label{eq-UV0}
\left\{\begin{aligned}
U_1(t) & :=a\,E_{\al,1}(-\eta_1t^\al)+\mu_1\int_0^t\tau^{\al-1}E_{\al,\al}(-\eta_1\tau^\al)F(t-\tau)\,\rd\tau,\\
V_1(t) & :=b\,E_{\be,1}(-\eta_2t^\be)+\mu_2\int_0^t\tau^{\be-1}E_{\be,\be}(-\eta_2\tau^\be)G(t-\tau)\,\rd\tau.
\end{aligned}\right.
\end{equation}
Starting from $(U_0,V_0)=(0,0)$, we iteratively define a sequence $\{(U_m,V_m)\}_{m=2}^\infty$ by
\begin{equation}\label{eq-UVm}
\left\{\begin{aligned}
U_{m+1}(t) & =U_1(t)+\mu_1\int_0^t\tau^{\al-1}E_{\al,\al}(-\eta_1\tau^\al)V_m(t-\tau)\,\rd\tau,\\
V_{m+1}(t) & =V_1(t)+\mu_2\int_0^t\tau^{\be-1}E_{\be,\be}(-\eta_2\tau^\be)U_m(t-\tau)\,\rd\tau,
\end{aligned}\right.\quad m=0,1,\dots.
\end{equation}
Our strategy is to show the convergence of $\{U_m\}_{m=0}^\infty$ and $\{V_m\}_{m=0}^\infty$ under some suitable norm, whose limit is indeed a solution of \eqref{eq-gov-3}. To this end, we denote $\bm U_m:=(U_m,V_m)^\T$, $m=0,1,\dots$ and define the operator $K$ as follows:
$$
K (U(t),V(t))^{\rm T} := \left(\mu_1 \int_0^t \tau^{\alpha-1} E_{\alpha,\alpha}(-\eta_1 \tau^\alpha) V(t-\tau)\,\rd\tau, \mu_2 \int_0^t \tau^{\beta-1} E_{\beta,\beta}(-\eta_2\tau^\beta) U(t-\tau)\,\rd\tau \right)^{\rm T}.
$$
Therefore, \eqref{eq-UVm} can be rephrased as
$$
\bm U_{m+1}(t)=\bm U_1(t)+K\bm U_m(t),\quad t\in \BR_+.
$$
Next, we shall give some estimates for $\bm U_m$ in some appropriate function space. Letting $T>0$ be fixed arbitrarily, by means of Lemma \ref{lem-ML}, we assert that $\bm U_m\in L^2(0,T)$ for $m=0,1,\dots$ and there exists a constant $C_K$ such that
\begin{equation}\label{eq-est-um}
|(\bm U_{m+1}-\bm U_m)(t)|\le C_K^{m-1}J^{(m-1)\be}|\bm U_1(t)|,\quad t\in(0,T).
\end{equation}

We will prove \eqref{eq-est-um} inductively. For $m=1$, from the definition of the function $\bm U_m$, the above estimate is trivial. By noting the assumption on the coefficients and the estimate in Lemma \ref{lem-ML}, we can directly obtain
\begin{align*}
|\bm U_1(t)| & \le\f{C|a|}{1+\eta_1t^\al}+\f{C|b|}{1+\eta_2t^\be}+C\mu_1\int_0^t\f{\tau^{\al-1}}{{1+\eta\tau^\al}}|f(t-\tau)|\,\rd\tau+C\mu_2\int_0^t\f{\tau^{\be-1}}{1+\eta\tau^\be}|g(t-\tau)|\,\rd\tau\\
& \le C\left(|a|+|b|+\int_0^t \tau^{\beta-1} (| f(t-\tau)| + |g(t-\tau)|\,\rd\tau \right)\\
& \le C(|a|+|b|)+C_K J^\beta (|f| + |g|)(t),
\quad t\in(0,T),
\end{align*}
where the last inequality is due to the fact that $t^{\al-1}<T^{\al-\be}t^{\be-1}$ if $\al>\be$ and $t\in(0,T)$ and the constants $C,C_K>0$ change by lines and depend on $\eta_1,\eta_2,\mu_1,\mu_2,\alpha,\beta$ and $T$. We combine the above inequalities with Young's convolution inequality to derive $\bm U_1\in L^2(0,T)$. 

Next, for any $m=2,3,\dots$ and $\bm U\in L^2(0,T)$, we can follow the above treatment to get
\begin{align*}
|K\bm U(t)| &\le  C_K J^\beta |\bm U(t)|,
\end{align*}
where the constant $C_K>0$ depends on $\eta_1,\eta_2,\mu_1,\mu_2,\alpha,\beta$ and $T$. Now we use the inductive assumption to derive
\begin{align*}
|(\bm U_{m+1}-\bm U_m)(t)| & =|K(\bm U_m-\bm U_{m-1})(t)|\\
& \le C_K J^{\be-1}C_K^{m-2}J^{(m-2)\be}|\bm U_1(t)|\\
& =C_K^{m-1}J^{(m-1)\be}|\bm U_1(t)|
\end{align*}
in view of the semi-group property $J^{\ga_1}J^{\ga_2}=J^{\ga_1+\ga_2}$ for any $\ga_1,\ga_2>0$. By the induction argument, we see that the estimate \eqref{eq-est-um} holds for any $m=1,2,\dots$. Again, from Young's convolution inequality, it follows that $\bm U_m=\sum_{j=0}^m(\bm U_{j+1}-\bm U_j)\in L^2(0,T)$ for any $m=1,2,\dots$. 

Moreover, for any integer $m_0$ satisfying $m_0\be>1/2$, on the basis of the above estimates, we further see that 
\begin{align*}
|(\bm U_{m+1}-\bm U_m)(t)| & \le C_K^{m-1}J^{(m-1)\be}\left(C(|a|+|b|)+C_K J^\be(|f|+|g|)(t)\right)\\
& =C(|a|+|b|)\f{C_K^{m-1}t^{(m-1)\be}}{\Ga((m-1)\be)}+C_K^m J^{(m-m_0)\be}J^{m_0\be}(|f|+|g|)(t),
\end{align*}
which, combined with the estimate
$$
J^{m_0\be}(|f|+|g|)(t)\le C\,t^{m_0\be-1/2}(\|f\|+\|g\|),\quad t\in(0,T)
$$
in view of the Cauchy-Schwartz inequality, implies that
\[
|(\bm U_{m+1}-\bm U_m)(t)|\le C(|a|+|b|)\f{C_K^{m-1}t^{(m-1)\be}}{\Ga(m\be-\be)}+C(\|f\|+\|g\|)\f{C_K^m\Ga(m_0\be+1/2)t^{m\be-1/2}}{\Ga(m\be+1/2)}
\]
for $m\ge m_0$. Now from the above estimates and the classical asymptotics
\[
\Ga(\eta)=\e^{-\eta}\eta^{\eta-1/2}(2\pi)^{1/2}\left(1+O\left(\f1\eta\right)\right)\quad\mbox{as }\eta\to+\infty
\]
(e.g., Abramowitz and Stegun \cite{AS72}, p.257), it follows that the series $\sum_{m=m_0}^\infty(\bm U_{m+1}-\bm U_m)$ is uniformly convergent. Now we assume that the limit of $\bm U_m$ is $\bm U$ as $m\to\infty$. Moreover, it is not difficult to check that $\bm U_1\ge\bm0$, that is,
\begin{align*}
& a\,E_{\al,1}(-\eta_1t^\al)+\int_0^t\tau^{\al-1}E_{\al,\al}(-\eta_1\tau^\al)F(t-\tau)\,\rd\tau\ge0,\\
& b\,E_{\be,1}(-\eta_2t^\be)+\int_0^t\tau^{\be-1}E_{\be,\be}(-\eta_2\tau^\be)G(t-\tau)\,\rd\tau\ge0,
\end{align*}
which implies $\bm U_m\ge\bm0$ for any $t\in(0,T]$ and hence the limit function $\bm U\ge\bm0$. 

Finally, we show that $\bm U$ must be strictly positive in $(0,T)$ if $a>0$. For this, we note that $\bm U$ satisfies the integral equation
$$
\bm U(t)=\bm U_1(t)+K\bm U(t),\quad t\in(0,T).
$$
Moreover, from the facts in Lemma \ref{lem-ML}, it follows that $a\,E_{\al,1}(-\eta_1t^\al)>0$. Collecting all the above results, we conclude that $\bm U(t)>\bm0$ for any $t>0$. The proof of the lemma is complete.
\end{proof}

\subsection{Asymptotic behavior of fractional ordinary system}

Next, we concentrate on a special case of \eqref{eq-gov-3}:
\begin{equation}\label{eq-gov-2}
\begin{cases}
\pa_t^\al(U-1)+c_1U-c_2V=0,\\
\pa_t^\be V-c_2U+c_1V=0
\end{cases}\mbox{in }\BR_+,
\end{equation}
that is,
\[
a=1,\quad b=F=G=0,\quad\eta_1=\eta_2=c_1,\quad\mu_1=\mu_2=c_2
\]
in \eqref{eq-gov-3}. Then according to Lemma \ref{lem-max}, the solution $(U,V)$ to \eqref{eq-gov-2} should be strictly positive in $\BR_+$ if $c_1\ge0,c_2>0$. In the next lemma, we provide a long-time decay estimate for \eqref{eq-gov-2} under a certain assumption on $c_1,c_2$.

\begin{lemma}\label{lem-asymp}
If $c_1>c_2>0,$ then there exists a constant $C>0$ depending only on $\al,\be,c_1,c_2$ such that the unique solution $(U,V)$ to \eqref{eq-gov-2} satisfies
$$
0<U(t)+V(t)\le\begin{cases}
C\,t^{-\al}, & \al<1,\\
C\,t^{-(1+\be)}, & \al=1,
\end{cases}\quad\forall\,t>1.
$$
\end{lemma}

\begin{proof}
Taking the Laplace transform $\wh f(s)=\int_{\BR_+}\e^{-s t}f(t)\,\rd t$ on both sides of \eqref{eq-gov-2} yields
\[
\begin{cases}
s^\al\wh U+c_1\wh U-c_2\wh V=s^{\al-1},\\
s^\be\wh V+c_1\wh V-c_2\wh U=0
\end{cases}
\]
and hence
\begin{equation}\label{eq-LapUV}
\wh U(s)=\f{s^{\al-1}(s^\be+c_1)}{(s^\al+c_1)(s^\be+c_1)-c_2^2},\quad\wh V(s)=\f{c_2s^{\al-1}}{(s^\al+c_1)(s^\be+c_1)-c_2^2}.
\end{equation}
By taking the inverse Laplace transform, we obtain
$$
U(t)=\f1{2\pi\,\ri}\int_{s_0-\ri\,\infty}^{s_0+\ri\,\infty}\wh U(s)\,\e^{s t}\,\rd s,\quad V(t)=\f1{2\pi\,\ri}\int_{s_0-\ri\,\infty}^{s_0+\ri\,\infty}\wh V(s)\,\e^{s t}\,\rd s,
$$
where $s_0>0$ is some constant.

It is easy to verify that the denominator of $\wh U(s),\wh V(s)$ are multi-valued on the complex plane. Moreover, the roots of the denominator of $\wh U(s),\wh V(s)$ are distributed in the left half of the complex plane excluding the negative real axis. Thus we can cut through the negative real axis to get a single valued holomorphic function. Denoting the poles of $\wh U(s),\wh V(s)$ as $z_1,\dots,z_N$, we take advantage of Lemma \ref{lem-residue1} to deduce
\begin{align*}
U(t) & =\sum_{k=1}^N\res_{s=z_k}\wh U(s)\,\e^{s t}+\f1{2\pi\,\ri}\oint_\ga\wh U(s)\,\e^{s t}\,\rd s=:I_1(t)+I_2(t),\\
V(t) & =\sum_{k=1}^N\res_{s=z_k}\wh V(s)\,\e^{s t}+\f1{2\pi\,\ri}\oint_\ga\wh V(s)\,\e^{s t}\,\rd s=:I_3(t)+I_4(t).
\end{align*}
It is obvious that $I_1(t)$ and $I_3(t)$ are of exponential decay as $t\to+\infty$. For $I_2(t)$ and $I_4(t)$, it can be easily proved that
\begin{align}
I_2(t)=\f1\pi\int_{\BR_+}\rIm\left[\wh U(s)\,\e^{s t}\right]_{s=r\,\e^{\ri\,\pi}}\rd r=\f1\pi\int_{\BR_+}\e^{-r t}\,\rIm\left[\wh U(s)\right]_{s=r\,\e^{\ri\,\pi}}\rd r,\label{eq-I2}\\
I_4(t)=\f1\pi\int_{\BR_+}\rIm\left[\wh V(s)\,\e^{s t}\right]_{s=r\,\e^{\ri\,\pi}}\rd r=\f1\pi\int_{\BR_+}\e^{-r t}\,\rIm\left[\wh V(s)\right]_{s=r\,\e^{\ri\,\pi}}\rd r.\label{eq-I4}
\end{align}
Since $\wh V(s)$ is slightly simpler than $\wh U(s)$, we shall deal with $I_4(t)$ first and then $I_2(t)$.\medskip

To deal with the integrand of $I_4(t)$, the substitution of \eqref{eq-LapUV} into $\rIm[\wh V(s)]_{s=r\,\e^{\ri\,\pi}}$ results in
\begin{align}
\rIm\left[\wh V(s)\right]_{s=r\,\e^{\ri\,\pi}} & =\rIm\left[\f{c_2s^{\al-1}}{(s^\al+c_1)(s^\be+c_1)-c_2^2}\right]_{s=r\,\e^{\ri\,\pi}}\nonumber\\
& =c_2\,\rIm\left[\f{r^{\al-1}\,\e^{\ri\,(\al-1)\pi}}{\{r^\al(\cos\al\pi+\ri\sin\al\pi)+c_1\}\{r^\be(\cos\be\pi+\ri\sin\be\pi)+c_1\}-c_2^2}\right]\nonumber\\
& =-c_2r^{\al-1}\,\rIm\left[\f{\e^{\ri\,\al\pi}}{q(r)}\right]=-c_2r^{\al-1}\f{\rIm(\e^{\ri\,\al\pi}\ov{q(r)})}{|q(r)|^2},\label{eq-imV}
\end{align}
where
\begin{align}
q(r) & =\{(c_1+r^\al\cos\al\pi)+\ri\,r^\al\sin\al\pi\}\{(c_1+r^\be\cos\be\pi)+\ri\,r^\be\sin\be\pi\}-c_2^2\nonumber\\
& =\left\{(c_1^2-c_2^2)+c_1(r^\al\cos\al\pi+r^\be\cos\be\pi)+r^{\al+\be}\cos(\al+\be)\pi\right\}\nonumber\\
& \quad\,+\ri\left\{c_1(r^\al\sin\al\pi+r^\be\sin\be\pi)+r^{\al+\be}\sin(\al+\be)\pi\right\}.\label{eq-def-q}
\end{align}
We claim that there exists a constant $\de>0$ such that
\begin{equation}\label{eq-pos-q}
|q(r)|^2\ge\de>0,\quad\forall\,r\ge0.
\end{equation}
In fact, since $\wh V(s)$ has no pole on the negative real axis, it is impossible for $q(r)$ to have zero points in $\BR_+$. For $r=0$, it follows from the assumption $c_1>c_2>0$ that $|q(0)|^2=(c_1^2-c_2^2)^2>0$. Meanwhile, it is readily seen from the above expression \eqref{eq-def-q} of $q(r)$ that
\[
|q(r)|^2=O(r^{2(\al+\be)})\longrightarrow+\infty\quad\mbox{as }r\to+\infty.
\]
These facts, together with the continuity of $q(r)$, indicate \eqref{eq-pos-q} immediately.

Next, we further calculate
\begin{align*}
\rIm\left(\e^{\ri\,\al\pi}\ov{q(r)}\right) & =\sin\al\pi\left\{(c_1^2-c_2^2)+c_1(r^\al\cos\al\pi+r^\be\cos\be\pi)+r^{\al+\be}\cos(\al+\be)\pi\right\}\\
& \quad\,-\cos\al\pi\left\{c_1(r^\al\sin\al\pi+r^\be\sin\be\pi)+r^{\al+\be}\sin(\al+\be)\pi\right\}\\
& =(c_1^2-c_2^2)\sin\al\pi+c_1r^\be\sin(\al-\be)\pi-r^{\al+\be}\sin\be\pi
\end{align*}
and hence
\[
\left|\rIm\left(\e^{\ri\,\al\pi}\ov{q(r)}\right)\right|\le\begin{cases}
C(r^\be+r^{1+\be}), & \al=1,\\
C(1+r^\be+r^{\al+\be}), & \al<1,
\end{cases}\quad\forall\,r\ge0.
\]
Plugging the above inequality and \eqref{eq-pos-q} into \eqref{eq-imV} and then \eqref{eq-I4}, we can estimate $I_4(t)$ as
\begin{align*}
|I_4(t)| & \le\f1\pi\int_{\BR_+}\e^{-r t}\left|\rIm\left[\wh V(s)\right]_{s=r\,\e^{\ri\,\pi}}\right|\rd r\le\f{c_2}\pi\int_{\BR_+}\e^{-r t}r^{\al-1}\f{|\rIm(\e^{\ri\,\al\pi}\ov{q(r)})|}{|q(r)|^2}\,\rd r\\
& \le\f{c_2}{\pi\de}\int_{\BR_+}\e^{-r t}r^{\al-1}\left|\rIm\left(\e^{\ri\,\al\pi}\ov{q(r)}\right)\right|\rd r\le\left\{\begin{alignedat}{2}
& C\int_{\BR_+}\e^{-r t}(r^\be+r^{1+\be})\,\rd r, & \quad & \al=1,\\
& C\int_{\BR_+}\e^{-r t}r^{\al-1}(1+r^\be+r^{\al+\be})\,\rd r, & \quad & \al<1
\end{alignedat}\right.\\
& \le\begin{cases}
C\,t^{-(1+\be)}, & \al=1,\\
C\,t^{-\al}, & \al<1,
\end{cases}\quad\forall\,t>1,
\end{align*}
where the Gamma function is employed  to calculate, e.g.
\[
\int_{\BR_+}\e^{-r t}r^\be\,\rd r=\int_{\BR_+}\e^{-\tau}\left(\f\tau t\right)^\be\,\f{\rd\tau}t=\Ga(1+\be)\,t^{-(1+\be)}.
\]
Combining with the exponential decay of $I_3(t)$, we arrive at
\begin{equation}\label{eq-decay-V}
0<V(t)\le\begin{cases}
C\,t^{-(1+\be)}, & \al=1,\\
C\,t^{-\al}, & \al<1,
\end{cases}\quad\forall\,t>1,
\end{equation}
where the strict positivity of $V(t)$ is guaranteed by Lemma \ref{lem-max}.\medskip

On the other hand, the calculation of $I_2(t)$ is more complicated, but the methodology is identical. Substituting \eqref{eq-LapUV} into $\rIm[\wh U(s)]_{s=r\,\e^{\ri\,\pi}}$ leads to
\begin{align}
\rIm\left[\wh U(s)\right]_{s=r\,\e^{\ri\,\pi}} & =\rIm\left[\f{s^{\al-1}(s^\be+c_1)}{(s^\al+c_1)(s^\be+c_1)-c_2^2}\right]_{s=r\,\e^{\ri\,\pi}}\nonumber\\
& =r^{\al-1}\,\rIm\left[\f{(\cos\al\pi+\ri\sin\al\pi)\{r^\be(\cos\be\pi+\ri\sin\be\pi)+c_1\}}{\{r^\al(\cos\al\pi+\ri\sin\al\pi)+c_1\}\{r^\be(\cos\be\pi+\ri\sin\be\pi)+c_1\}-c_2^2}\right]\nonumber\\
& =-r^{\al-1}\,\rIm\left[\f{p(r)}{q(r)}\right]=-r^{\al-1}\f{\rIm(p(r)\ov{q(r)})}{|q(r)|^2},\label{eq-imU}
\end{align}
where $q(r)$ is the same as \eqref{eq-def-q} and
\begin{align*}
p(r) & =(\cos\al\pi+\ri\sin\al\pi)\left\{(c_1+r^\be\cos\be\pi)+\ri\sin\be\pi\right\}\\
& =\left\{c_1\cos\al\pi+r^\be\cos(\al+\be)\pi\right\}+\ri\left\{c_1\sin\al\pi+r^\be\sin(\al+\be)\pi\right\}.
\end{align*}
In a same manner as before, we calculate
\begin{align*}
& \quad\,\rIm\left(p(r)\ov{q(r)}\right)\\
& =\left\{c_1\sin\al\pi+r^\be\sin(\al+\be)\pi\right\}\left\{(c_1^2-c_2^2)+c_1(r^\al\cos\al\pi+r^\be\cos\be\pi)+r^{\al+\be}\cos(\al+\be)\pi\right\}\\
& \quad\,-\left\{c_1\cos\al\pi+r^\be\cos(\al+\be)\pi\right\}\left\{c_1(r^\al\sin\al\pi+r^\be\sin\be\pi)+r^{\al+\be}\sin(\al+\be)\pi\right\}\\
& =c_1(c_1^2-c_2^2)\sin\al\pi+\left\{c_1^2\sin(\al-\be)\pi+(c_1^2-c_2^2)\sin(\al+\be)\pi\right\}r^\be+c_1(\sin\al\pi)r^{2\be}
\end{align*}
and hence
\[
\left|\rIm\left(p(r)\ov{q(r)}\right)\right|\le\begin{cases}
C\,r^\be, & \al=1,\\
C(1+r^\be+r^{2\be}), & \al<1,
\end{cases}\quad\forall\,r\ge0.
\]
The above inequality is then combined with \eqref{eq-pos-q} and \eqref{eq-imU} to estimate $I_2(t)$ in \eqref{eq-I2} as
\begin{align*}
|I_2(t)| & \le\f1{\pi\de}\int_{\BR_+}\e^{-r t}r^{\al-1}\left|\rIm\left(p(r)\ov{q(r)}\right)\right|\rd r\le\left\{\begin{alignedat}{2}
& C\int_{\BR_+}\e^{-r t}r^\be\,\rd r, & \quad & \al=1,\\
& C\int_{\BR_+}\e^{-r t}r^{\al-1}(1+r^\be+r^{2\be})\,\rd r, & \quad & \al<1
\end{alignedat}\right.\\
& \le\begin{cases}
C\,t^{-(1+\be)}, & \al=1,\\
C\,t^{-\al}, & \al<1,
\end{cases}\quad\forall\,t>1.
\end{align*}
Following the same line as before, we can conclude \eqref{eq-decay-V} for $U(t)$, which completes the proof of Lemma \ref{lem-asymp}.
\end{proof}

\subsection{Completion of the proof of Theorem \ref{thm-asymp}}

Now we have collected all necessary ingredients to finish the proof of Theorem \ref{thm-asymp}.

In the sequel, we denote $U(t):=\|u(t)\|$ and $V(t):=\|v(t)\|$. Taking the $L^2(\Om)$ inner product on both sides of the first governing equation in \eqref{eq-gov} with $u$, we utilize Lemma \ref{lem-coer} and assumptions \eqref{eq-assume1}--\eqref{eq-assume2} to estimate
\begin{align*}
0 & =(u(t),\pa_t^\al(u(t)-u_0)-\rdiv(\bm A\nb u(t))+c_{11}(t)u(t)+c_{12}(t)v(t))\\
& \ge U(t)\,\pa_t^\al(U(t)-U(0))+\int_\Om\bm A\nb u(t)\cdot\nb u(t)\,\rd\bm x-\|c_{12}\|_{L^\infty(\Om\times\BR_+)}|(u(t),v(t))|\\
& \ge U(t)\,\pa_t^\al(U(t)-U(0))+\ka_0\|\nb u(t)\|^2-\ka_1|(u(t),v(t))|\\
& \ge U(t)\,\pa_t^\al(U(t)-U(0))+\f{\ka_0}{C_\Om^2}U^2(t)-\ka_1U(t)V(t),\quad t>0,
\end{align*}
where $C_\Om$ and $\ka_1$ are the constants introduced in \eqref{eq-Poincare} and \eqref{eq-assume0}, respectively. From $U(t)\ge0$, it follows that
\begin{equation}\label{ineq-U}
\pa_t^\al(U-U(0))+\f{\ka_0}{C_\Om^2}U-\ka_1V\le0\quad\mbox{in }\BR_+.
\end{equation}
Repeating the same procedure for the governing equation of $v$ in \eqref{eq-gov} implies that
\begin{equation}\label{ineq-V}
\pa_t^\be V+\f{\ka_0}{C_\Om^2}V-\ka_1U\le0\quad\mbox{in }\BR_+,
\end{equation}
where $V(0)=0$ by assumption.

As $(U,V)$ satisfies the coupled system \eqref{ineq-U}--\eqref{ineq-V} of fractional ordinary differential inequalities, for clarity, we introduce an auxiliary fractional ordinary differential system
\[
\left\{\begin{aligned}
& \pa_t^\al(\ov U-1)+\f{\ka_0}{C_\Om^2}\ov U-\ka_1\ov V=0,\\
& \pa_t^\al\ov V+\f{\ka_0}{C_\Om^2}\ov V-\ka_1\ov U=0
\end{aligned}\quad\mbox{in }\BR_+.
\right.
\]
Then Lemma \ref{lem-max} and the linearity of the problem indicate
\[
U\le U(0)\ov U,\quad V\le U(0)\ov V\quad\mbox{in }\BR_+.
\]
On the other hand, due to \eqref{eq-assume0}, the key assumption in Lemma \ref{lem-asymp} is satisfied, hence it can be inferred that
\[
\ov U(t)+\ov V(t)\le\begin{cases}
C\,t^{-\al}, & \al<1,\\
C\,t^{-(1+\be)}, & \al=1,
\end{cases}\quad\forall\,t>1.
\]
Consequently, the above two inequalities imply the desired decay estimate \eqref{eq-asymp1} by returning to the original notations $\|u(t)\|=U(t)$, $\|v(t)\|=V(t)$ and $\|u_0\|=U(0)$. The proof of Theorem \ref{thm-asymp} is completed.

\section{Numerical schemes and stability analysis}\label{sec-scheme}

In this section, we propose semi-implicit and fully implicit numerical schemes for the initial-boundary value problem \eqref{eq-gov}. The establishment of these schemes plays a fundamental role not only in further numerical analysis for coupled subdiffusion systems, but also in possible applications to numerical reconstructions of corresponding inverse problems. Especially in the current work, since we are scrutinizing the long-time asymptotic behavior of the solution, it is also essential to carry out basic stability analysis accordingly in order to guarantee the feasibility of the proposed methods. Actually, the theoretical findings above were indeed inspired by the numerical experiments demonstrated in the next section.

To concentrate on the nonlocal effect in time, without loss of generality we keep it simple on the spatial direction to consider the following problem:
\begin{equation}\label{eq-gov2}
\begin{cases}
\begin{aligned}
& \pa_t^\al(u-u_0)-d_1u_{xx}+c_{11}(x,t)u+c_{12}(x,t)v=F_1(x,t),\\
& \pa_t^\be(v-v_0)-d_2v_{xx}+c_{21}(x,t)u+c_{22}(x,t)v=F_2(x,t)
\end{aligned}
& \mbox{in }(0,L)\times(0,T),\\
u=v=0 & \mbox{on }\{0,L\}\times(0,T),
\end{cases}
\end{equation}
where $L,d_1,d_2$ are positive constants. In other words, we restrict the spatial dimension $d=1$ and the principal coefficients $A(x),B(x)$ as constants. Meanwhile, the appearance of inhomogeneous terms $F_1(x,t),F_2(x,t)$ is allowed for the completeness of the numerical schemes, which will be set as zero in the study of  asymptotic behaviors. Also, the discussion is limited to the case of $\al<1$, that is, both orders are fractional, as that of $\al=1$ is classical.

For the time interval $[0,T]$, we perform an equidistant partition with the step size $\De t=T/N$ ($N=2,3,\dots$), and denote the grid points as $t_n:=n\De t$ ($n=0,1,\dots,N$). For the numerical computation, we adopt the original Caputo derivative $\rd_t^\al$ (see \eqref{eq-Caputo}) instead of $\pa_t^\al$ and recall the L1 approximation of $\rd_t^\al$ introduced in Lin and Xu \cite{LX07} for $h\in C^1[0,T]$:
\begin{align*}
\pa^\al_t(h(t_{n+1})-h(0)) & =\rd_t^\al h(t_{n+1})=\f1{\Ga(1-\al)}\sum_{j=0}^n\int_{t_j}^{t_{j+1}}\f{h'(s)}{(t_{n+1}-s)^\al}\,\rd s\\
& =\f1{\Ga(1-\al)}\sum_{j=0}^n\f{h(t_{j+1})-h(t_j)}{\De t}\int_{t_j}^{t_{j+1}}\f{\rd s}{(t_{n+1}-s)^\al}+R_{\De t}^{n+1}\\
& =\f1{\Ga(2-\al)}\sum_{j=0}^n b_1^j\f{h(t_{n-j+1})-h(t_{n-j})}{\De t^\al}+R_{\De t}^{n+1},\quad n=0,\dots,N-1,
\end{align*}
where
\[
b_1^j:=(j+1)^{1-\al}-j^{1-\al},\quad j=0,1,\dots,n
\]
and $R_{\De t}^{n+1}$ stands for the truncation error. It was shown in \cite{LX07} that if $h\in C^2[0,T]$, then $R_{\De t}^{n+1}=O(\De t^{2-\al})$ uniformly for $n=1,2,\dots$.

On the spatial direction, we similarly discretize the interval $[0,L]$ with a step size $\De x=L/I$ ($I=2,3,\dots$) and write $x_i:=i\De x$ ($i=0,1,\dots,I$). For the Laplacian, we simply employ the central difference discretization. For $n=0,1,\dots,N$ and $i=0,1,\dots,I$, by $u^n_i$ and $v^n_i$ we denote the approximations of the true solutions $u$ and $v$ to \eqref{eq-gov2} at $(x,t)=(x_i,t_n)$, and abbreviate
\[
c_{k\ell}^{i,n}:=c_{k\ell}(x_i,t_n),\quad F_k^{i,n}:=F_k(x_i,t_n),\quad k,\ell=1,2.
\]
For the numerical computation, all functions $c_{k\ell},F_k$ are assumed to allow pointwise definitions at grid points.

Now we are well prepared to discretize the problem \eqref{eq-gov2}. Assume that we have finished the simulation of $u,v$ until $t=t_n$ and are in a position to proceed to the next level $t=t_{n+1}$ for some $n=0,\dots,N-1$. Regarding the lower order parts $-c_{k1}u-c_{k2}v+F_k$ ($k=1,2$) as given data at the previous level $t=t_n$, we propose a semi-implicit scheme for \eqref{eq-gov2} as
\begin{equation}\label{num-cal-discrete}
\begin{aligned}
\f1{\Ga(2-\al)}\sum_{j=0}^n b_1^j\f{u^{n-j+1}_i-u^{n-j}_i}{\De t^\al}-d_1\f{u^{n+1}_{i-1}-2u^{n+1}_i+u^{n+1}_{i+1}}{\De x^2} & =-c_{11}^{i,n}u^n_i-c_{12}^{i,n}v^n_i+F_1^{i,n},\\
\f1{\Ga(2-\be)}\sum_{j=0}^n b_2^j\f{v^{n-j+1}_i-v^{n-j}_i}{\De t^\be}-d_2\f{v^{n+1}_{i-1}-2v^{n+1}_i+v^{n+1}_{i+1}}{\De x^2} & =-c_{21}^{i,n}u^n_i-c_{22}^{i,n}v^n_i+F_2^{i,n},
\end{aligned}
\end{equation}
where
\[
b_2^j:=(j+1)^{1-\be}-j^{1-\be},\quad j=0,1,\dots,n.
\]
Further, introducing
\[
r_1:=\f{d_1\Ga(2-\al)\De t^\al}{\De x^2},\quad r_2:=\f{d_2\Ga(2-\be)\De t^\be}{\De x^2}, 
\]
we can rewrite \eqref{num-cal-discrete} by separating unknowns from computed data on both sides as
\begin{equation}\label{semi-imp}
\begin{aligned}
-r_1u^{n+1}_{i-1}+(1+2r_1)u^{n+1}_i-r_1u^{n+1}_{i+1} & =b_1^n u^0_i+\sum_{j=1}^n\left(b_1^{n-j}-b_1^{n-j+1}\right)u^j_i\\
& \quad\,+\f{\De x^2r_1}{d_1}\left(-c_{11}^{i,n}u^n_i-c_{12}^{i,n}v^n_i+F_1^{i,n}\right),\\
-r_2v^{n+1}_{i-1}+(1+2r_2)v^{n+1}_i-r_2v^{n+1}_{i+1} & =b_2^n v^0_i+\sum_{j=1}^n\left(b_2^{n-j}-b_2^{n-j+1}\right)v^j_i\\
& \quad\,+\f{\De x^2r_2}{d_2}\left(-c_{21}^{i,n}u^n_i-c_{22}^{i,n}v^n_i+F_2^{i,n}\right)
\end{aligned}
\end{equation}
for $n=0,\dots,N-1$ and $i=0,\dots,I$. Then the two equations in \eqref{semi-imp} are decoupled from each other and thus can be solved independently. Since the homogeneous Dirichlet boundary condition gives $u_0^n=u_I^n=v_0^n=v_I^n=0$ for all $n$, one can reformulate \eqref{semi-imp} as two linear systems with tridiagonal positive-definite matrices. Then the stability analysis for single subdiffusion equations works also in our case, and automatically the semi-implicit scheme \eqref{semi-imp} is absolutely stable.

Following the same line, we regard the lower order parts $-c_{k1}u-c_{k2}v+F_k$ ($k=1,2$) as unknowns at the next level $t=t_{n+1}$ to propose a fully implicit scheme for \eqref{eq-gov2} as
\begin{align*}
\f1{\Ga(2-\al)}\sum_{j=0}^n b_1^j\f{u^{n-j+1}_i-u^{n-j}_i}{\De t^\al}-d_1\f{u^{n+1}_{i-1}-2u^{n+1}_i+u^{n+1}_{i+1}}{\De x^2}+c_{11}^{i,n+1}u^{n+1}_i+c_{12}^{i,n+1}v^{n+1}_i & =F_1^{i,n+1},\\
\f1{\Ga(2-\be)}\sum_{j=0}^n b_2^j\f{v^{n-j+1}_i-v^{n-j}_i}{\De t^\be}-d_2\f{v^{n+1}_{i-1}-2v^{n+1}_i+v^{n+1}_{i+1}}{\De x^2}+c_{21}^{i,n+1}u^{n+1}_i+c_{22}^{i,n+1}v^{n+1}_i & =F_2^{i,n+1}.
\end{align*}
Again separating unknowns from computed data on both sides yields
\begin{equation}\label{full-imp0}
\begin{aligned}
& \quad\,-r_1u^{n+1}_{i-1}+\left(1+2r_1+\f{\De x^2r_1}{d_1}c_{11}^{i,n+1}\right)u^{n+1}_i-r_1u^{n+1}_{i+1}+\f{\De x^2r_1}{d_1}c_{12}^{i,n+1}v^{n+1}_i\\
& =b_1^n u^0_i+\sum_{j=1}^n(b_1^{n-j}-b_1^{n-j+1})u^j_i+\f{\De x^2r_1}{d_1}F_1^{i,n+1},\\
& \quad\,-r_2v^{n+1}_{i-1}+\left(1+2r_2+\f{\De x^2r_2}{d_2}c_{22}^{i,n+1}\right)v^{n+1}_i-r_2v^{n+1}_{i+1}+\f{\De x^2r_2}{d_2}c_{21}^{i,n+1}u^{n+1}_i\\
& =b_2^n v^0_i+\sum_{j=1}^n(b_2^{n-j}-b_2^{n-j+1})v^j_i+\f{\De x^2r_2}{d_2}F_2^{i,n+1}.
\end{aligned}
\end{equation}
We denote
\begin{align*}
\bm u^n & :=\left(u^n_1,\dots,u^n_{I-1},v^n_1,\dots,v^n_{I-1}\right)^\T\in\BR^{2(I-1)},\\
\bm F^n & :=\left(F_1^{1,n},\dots,F_1^{I-1,n},F_2^{1,n},\dots,F_2^{I-1,n}\right)^\T\in\BR^{2(I-1)}
\end{align*}
and abbreviate the coefficients appearing above as
\[
\begin{aligned}
\al_i & :=1+2r_1+\f{\De x^2r_1}{d_1}c_{11}^{i,n+1}, & \quad\ga^1_i & :=\f{\De x^2r_1}{d_1}c_{12}^{i,n+1},\\
\be_i & :=1+2r_2-\f{\De x^2r_2}{d_2}c_{22}^{i,n+1}, & \quad\ga^2_i & :=\f{\De x^2r_2}{d_2}c_{21}^{i,n+1}
\end{aligned}\quad(i=1,\dots,I-1).
\]
Then \eqref{full-imp0} can be reformulated as a linear system
\begin{equation}\label{full-imp1}
\rA\bm u^{n+1}=\rB^n_0\bm u^0+\sum_{j=1}^n\rB^{n-j}\bm u^j+\rC\bm F^{n+1},
\end{equation}
where the coefficient matrix $\rA$ is a block matrix
\[
\rA:=\begin{pmatrix} 
\rA_{11} & \rA_{12}\\
\rA_{21} & \rA_{22}  
\end{pmatrix}\in\BR^{2(I-1)\times2(I-1)}
\]
with
\begin{align*}
\rA_{11} & :=\begin{pmatrix} 
\al_1 & -r_1 &&&&&& \text{\huge0}\\
-r_1 & \al_2 && -r_1\\  
& \ddots && \ddots && \ddots\\ 
&&& -r_1  &&\al_{I-2} && -r_1\\  
\text{\huge0} &&&&& -r_1  &&\al_{I-1}\\
\end{pmatrix}\in\BR^{(I-1)\times(I-1)},\\
\rA_{22} & :=\begin{pmatrix} 
\be_1 & -r_2 &&&&&& \text{\huge0}\\
-r_2 & \be_2 && -r_2\\  
& \ddots && \ddots && \ddots\\ 
&&& -r_2  &&\be_{I-2} && -r_2\\  
\text{\huge0} &&&&& -r_2  &&\be_{I-1} 
\end{pmatrix}\in\BR^{(I-1)\times(I-1)},\\
\rA_{12} & :=-\diag(\ga^1_1,\dots,\ga^1_{I-1})\in\BR^{(I-1)\times(I-1)},\\
\rA_{21} & :=-\diag(\ga^2_1,\dots,\ga^2_{I-1})\in\BR^{(I-1)\times(I-1)}.
\end{align*}
The matrices on the right-hand side of \eqref{full-imp1} are
\begin{align*}
\rB^n_0 & :=\diag(b_1^n,\dots,b_1^n,b_2^n,\dots,b_2^n)\in\BR^{2(I-1)\times2(I-1)},\\
\rB^j & :=\diag(b_1^j-b_1^{j+1},\dots,b_1^j-b_1^{j+1},b_2^j-b_2^{j+1},\dots,b_2^j-b_2^{j+1})\in\BR^{2(I-1)\times2(I-1)},\quad j=0,\dots,n-1,\\
\rC & :=\diag\left(\f{\De x^2r_1}{d_1},\dots,\f{\De x^2r_1}{d_1},\f{\De x^2r_2}{d_2},\dots,\f{\De x^2r_2}{d_2}\right)\in\BR^{2(I-1)\times2(I-1)}.
\end{align*}
This completes the construction of the fully implicit scheme for \eqref{eq-gov2}.

Next, we analyze the numerical stability of the above proposed scheme \eqref{full-imp1}. To this end, it suffices to verify that the absolute values of all eigenvalues of the coefficient matrix $\rA$ in \eqref{full-imp1} are greater than $1$. Denoting the set of all eigenvalues of $\rA$ by $\si(\rA)$, we take advantage of the Gershgorin circle theorem to deduce
\[
\si(\rA)\subset\bigcup_{i=1}^{2(I-1)}\rR_i,\
\]
where
\[
\rR_i:=\begin{cases}
\left\{z\in\BC\mid|z-\al_i|\le r_1+|\ga^1_i|\right\}, & i=1,I-1,\\
\left\{z\in\BC\mid|z-\al_i|\le2r_1+|\ga^1_i|\right\}, & i=2,\dots,I-2,\\
\left\{z\in\BC\mid|z-\be_i|\le r_2+|\ga^2_i|\right\}, & i=I,2I-2,\\
\left\{z\in\BC\mid|z-\be_i|\le2r_2+|\ga^2_i|\right\}, & i=I+1,\dots,2I-3,
\end{cases}
\]
We represent disks $\rR_i,\rR_{i+I-1}$ ($i=2,\dots,I-2$) in detail as
\begin{align*}
\rR_i:|z-\al_i| & =\left|z-\left(1+2r_1+\f{\De x^2r_1}{d_1}c_{11}^{i,n+1}\right)\right|\le2r_1+|\ga^1_i|=2r_1+\f{\De x^2r_1}{d_1}|c_{12}^{i,n+1}|,\\
\rR_{i+I-1}:|z-\be_i| & =\left|z-\left(1+2r_2+\f{\De x^2r_2}{d_2}c_{22}^{i,n+1}\right)\right|\le2r_2+|\ga^2_i|=2r_2+\f{\De x^2r_2}{d_2}|c_{21}^{i,n+1}|.
\end{align*}
Similarly, for $\rR_1,\rR_{I-1},\rR_I,\rR_{2(I-1)}$, it can be shown that
\begin{align*}
\rR_1:|z-\al_1| & =\left|z-\left(1+2r_1+\f{\De x^2r_1}{d_1}c_{11}^{1,n+1}\right)\right|\le r_1+|\ga^1_1|=r_1+\f{\De x^2r_1}{d_1}|c_{12}^{1,n+1}|,\\
\rR_{I-1}:|z-\al_{I-1}| & =\left|z-\left(1+2r_1+\f{\De x^2r_1}{d_1}c_{11}^{I-1,n+1}\right)\right|\le r_1+|\ga^1_{I-1}|=r_1+\f{\De x^2r_1}{d_1}|c_{12}^{I-1,n+1}|,\\
\rR_I:|z-\be_1| & =\left|z-\left(1+r_2+\f{\De x^2r_2}{d_2}c_{22}^{1,n+1}\right)\right|\le r_2+\ga^2_1|=r_2+\f{\De x^2r_2}{d_2}|c_{21}^{1,n+1}|,\\
\rR_{2(I-1)}:|z-\be_{I-1}| & =\left|z-\left(1+2r_2+\f{\De x^2r_2}{d_2}c_{22}^{I-1,n+1}\right)\right|\le r_2+|\ga^2_{I-1}|=r_2+\f{\De x^2r_2}{d_2}|c_{21}^{I-1,n+1}|.
\end{align*}
Recalling the global assumption \eqref{eq-assume2} on the coupling coefficients $c_{k\ell}$, we see that if
\begin{equation}\label{stab-cond}
c_{11}\ge|c_{12}|,\quad c_{22}\ge|c_{21}|\quad\mbox{in }\Om\times(0,T),
\end{equation}
then
\[
2r_1+|\ga^1_i|-|\al_i|\ge1,\quad2r_2+|\ga^2_i|-|\be_i|\ge1,\quad i=1,\dots,I-1
\]
and hence the domain $\rS_\rR$ lies completely outside the open unit ball $\{|z|<1\}$, indicating that the absolute values of all eigenvalues of the coefficient matrix $\rA$ are no less than $1$. Therefore, we conclude that \eqref{stab-cond} is a sufficient condition for the absolute stability of the fully implicit scheme \eqref{full-imp1}. The condition \eqref{stab-cond} is in principle independent of the key assumption \eqref{eq-assume0} in Theorem \ref{thm-asymp}, but is highly related to the semi-definiteness condition in \cite[Theorem 2]{LHL23}.

We close this section by giving several remarks. First, we notice that the fully implicit scheme \eqref{full-imp1} can only be easily implemented for linear systems. Even for the simplest nonlinear systems such as the following semilinear subdiffusion-reaction system
\[
\begin{cases}
\begin{aligned}
& \pa_t^\al(u-u_0)-d_1u_{xx}=G_1(u,v),\\
& \pa_t^\be(v-v_0)-d_2v_{xx}=G_2(u,v)
\end{aligned}
& \mbox{in }(0,L)\times(0,T),\\
u=v=0 & \mbox{on }\{0,L\}\times(0,T),
\end{cases}
\]
one should take advantage of some Newton-type methods to solve a coupled nonlinear system at each step to implement a fully implicit scheme. In this sense, nonlinear generalizations of the semi-implicit scheme \eqref{semi-imp} turn out to be rather convenient in applications. On the other hand, it reveals in the numerical tests in the next section that both semi-implicit and fully implicit schemes demonstrate high numerical accuracy and there seems no mentionable difference between their numerical performance even in observing the long-time asymptotic behavior of solutions.

Finally, we mention that all above discussions automatically work for more components than $2$, e.g., in the next section we will also deal with coupled systems with $3$ components. Likewise, the coupling of subdiffusion and usual diffusion equations can also be discretized in the same manner, and the absolute stability still holds as long as we keep the scheme implicit. The arguments are almost identical and the details here are omitted.

\section{Numerical verification of decay rates}\label{sec-test}

Based on the numerical schemes proposed in the previous section, this section is devoted to the numerical justification of the large time asymptotic behavior of solutions to the initial-boundary value problem \eqref{eq-gov2} as well as its 3-component counterpart, which hopefully provide motivative hints to the theoretical studies of the sharp decay rates of solutions.

For \eqref{eq-gov2}, first we fix the coefficients in the governing equations as
\[
L=\pi,\quad\quad d_1=d_2=1,\quad c_{11}=c_{22}=1,\quad c_{12}=c_{21}=-1
\]
and the source terms as $F_1=F_2=0$. As for the initial values, we are concerned with whether $v_0$ vanishes identically or not. In numerical examples, the following cases are considered.
\begin{equation}\label{eq-IC2}
\left\{\!\begin{aligned}
\mbox{(i)}\ & u_0(x)=\sin x,\ v_0(x)=\f\pi2-\left|x-\f\pi2\right|,\\
\mbox{(ii)}\ & u_0(x)=\sin x,\ v_0(x)=0.
\end{aligned}\right.
\end{equation}
We explain the method for verifying the numerical decay rates of solutions. If the $L^2(\Om)$-norms of $u(x,t)$ and $v(x,t)$ decay at a rate of a constant power of $t$, then the functions
\begin{equation}\label{eq-numer-decay}
\f{\log\|u(t)\|}{\log t},\quad\f{\log\|v(t)\|}{\log t}
\end{equation}
should converge to that constant as $t\to+\infty$. Therefore, after computing the numerical solutions of $u,v$, we approximate $\|u(t)\|$ and $\|v(t)\|$ by some numerical integration and then evaluate the above quantities \eqref{eq-numer-decay} until a sufficiently large final time $T$. In the following examples, we basically choose $T=1000$, which turns out to be sufficient to conclude the numerical convergence of the decay rates.

\begin{example}\label{ex-2frac}
First we choose the orders of time derivatives in \eqref{eq-gov2} as
\[
\al=0.9,\quad\be=0.5.
\]
For both cases in \eqref{eq-IC2}, we perform numerical simulations by the schemes proposed in the previous section and compute the quantities \eqref{eq-numer-decay} to verify the decay rates of both components of the solution. The numerical results are illustrated in Figure \ref{asym-be_TFDS_K=2}, which shows clear convergence to expected constants highlighted by dashed red lines. In detail, we can conclude
\[
\left\{\begin{aligned}
\mbox{(i)}\ & \mbox{If $u_0\not\equiv0$ and $v_0\not\equiv0$, then }\|u(t)\|,\|v(t)\|\sim t^{-0.5},\\
\mbox{(ii)}\ & \mbox{If $u_0\not\equiv0$ and $v_0\equiv0$, then }\|u(t)\|,\|v(t)\|\sim t^{-0.9}
\end{aligned}\right.\quad\mbox{as }t\to+\infty.
\]
The result in Case (i) confirms the sharp decay rate $t^{-\be}$ obtained in Lemma \ref{lem-wp}(ii), while that in Case (ii) obviously suggests the decay rate $t^{-\al}$ in the case of $\al<1$, which was later established in Theorem \ref{thm-asymp}. Intuitively, due to the lack of the initial supply, the time evolution of the component $v$ corresponding to the smaller order $\be$ relies completely on the supply from the other component $u$. As a result, the long-time asymptotic behavior of $v$ reflects the decay rate $t^{-\al}$ of $u$.
\begin{figure}[htbp]
\includegraphics[width=0.48\textwidth]{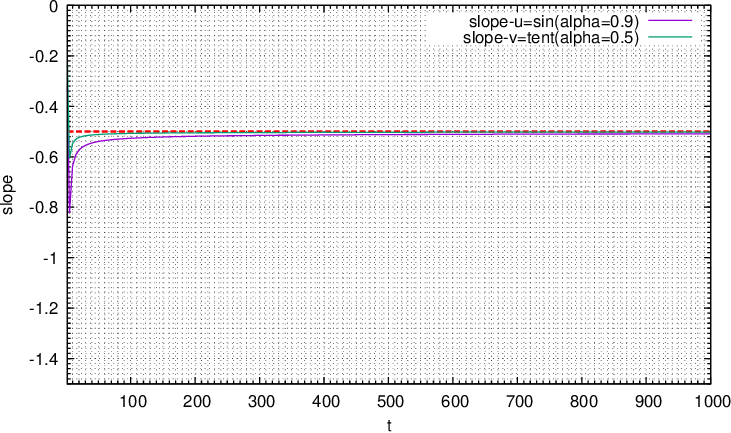}\quad
\includegraphics[width=0.48\textwidth]{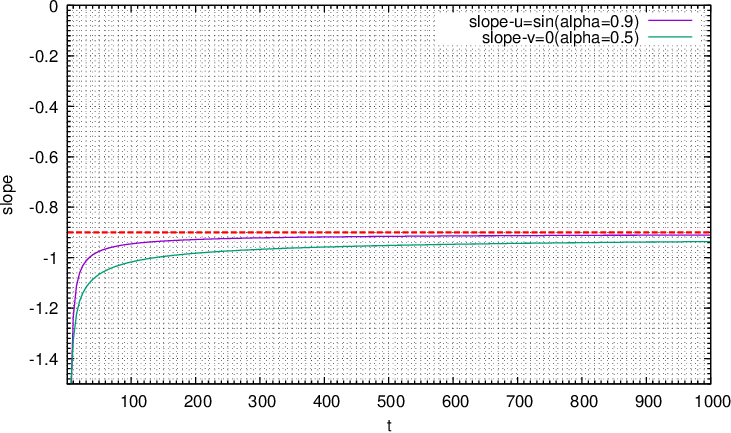}
\caption{Long-time decay rates of solutions to the coupled system \eqref{eq-gov2} with $\al=0.9$ and $\be=0.5$. Left: Case (i) of the initial values. Right: Case (ii) of the initial values.}\label{asym-be_TFDS_K=2}
\end{figure}
\end{example}

\begin{example}
Next, we study the more interesting case of $\al=1$ in \eqref{eq-gov2} and change several different values of $\be<1$. We first fix $\be=0.5$ and repeat the same test in Example \ref{ex-2frac}. From the numerical results plotted in Figure \ref{asym-be_TPFDS}, it can be clearly observed that
\[
\left\{\begin{aligned}
\mbox{(i)}\ & \mbox{If $u_0\not\equiv0$ and $v_0\not\equiv0$, then }\|u(t)\|,\|v(t)\|\sim t^{-0.5},\\
\mbox{(ii)}\ & \mbox{If $u_0\not\equiv0$ and $v_0\equiv0$, then }\|u(t)\|,\|v(t)\|\sim t^{-1.5}
\end{aligned}\right.\quad\mbox{as }t\to+\infty.
\]
As expected, the result in Case (i) complies with the same sharp decay rate $t^{-\be}$ as before, indicating that Lemma \ref{lem-wp}(ii) still holds true for $\al=1$. However, the superlinear decay observed in Case (ii) is rather unexpected and surprising, which seems never happen in the coupling of two usual diffusion equations or two subdiffusion equations.
\begin{figure}[htbp]
\includegraphics[width=0.48\textwidth]{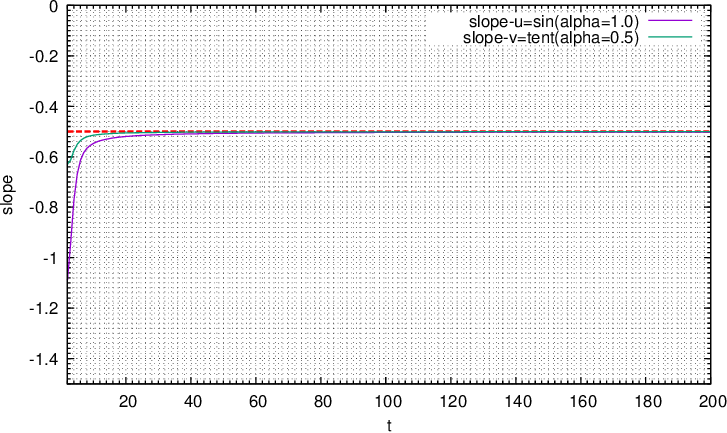}\quad
\includegraphics[width=0.48\textwidth]{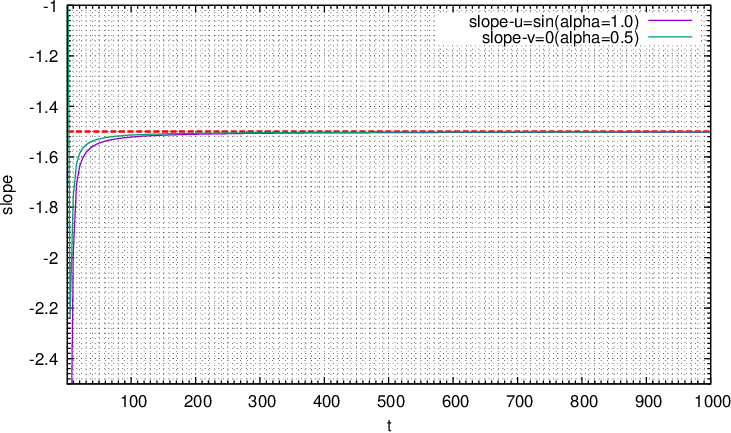}
\caption{Long-time decay rates of solutions to the coupled system \eqref{eq-gov2} with $\al=1.0$ and $\be=0.5$. Left: Case (i) of the initial values. Right: Case (ii) of the initial values.}\label{asym-be_TPFDS}
\end{figure}
 
In order to deepen the understanding of this decay pattern, we fix the choice of initial values as Case (ii) in \eqref{eq-IC2} (i.e., $u_0\not\equiv0$ and $v_0\equiv0$) and change the value of $\be$ to identify the dependency of the decay rate on $\be$. We choose $\be=0.3$, $\be=0.7$ and repeat the same procedure as before. As is shown in Figure \ref{asym-be_TPFDS_K=2}, we can observe
\[
\left\{\begin{aligned}
& \mbox{If $\al=1.0$ and $\be=0.3$, then }\|u(t)\|,\|v(t)\|\sim t^{-1.3},\\
& \mbox{If $\al=1.0$ and $\be=0.7$, then }\|u(t)\|,\|v(t)\|\sim t^{-1.7}
\end{aligned}\right.\quad\mbox{as }t\to+\infty.
\]
Evidently, it is reasonable to conjecture that the decay rate of both components of the solution is $t^{-(1+\be)}$ in the case of $\al=1$ and $v_0\equiv0$. The condition for realizing such a unique pattern turns out to be somehow restrictive, namely, the coupling should be a mixture of fractional and non-fractional equations and the initial value of the latter should vanish. In such a sense, this superlinear decay reflects a subtle balance between exponential and sublinear decays, which is, fortunately, theoretically demonstrated in Theorem \ref{thm-asymp}.
\begin{figure}[htbp]
\includegraphics[width=0.48\textwidth]{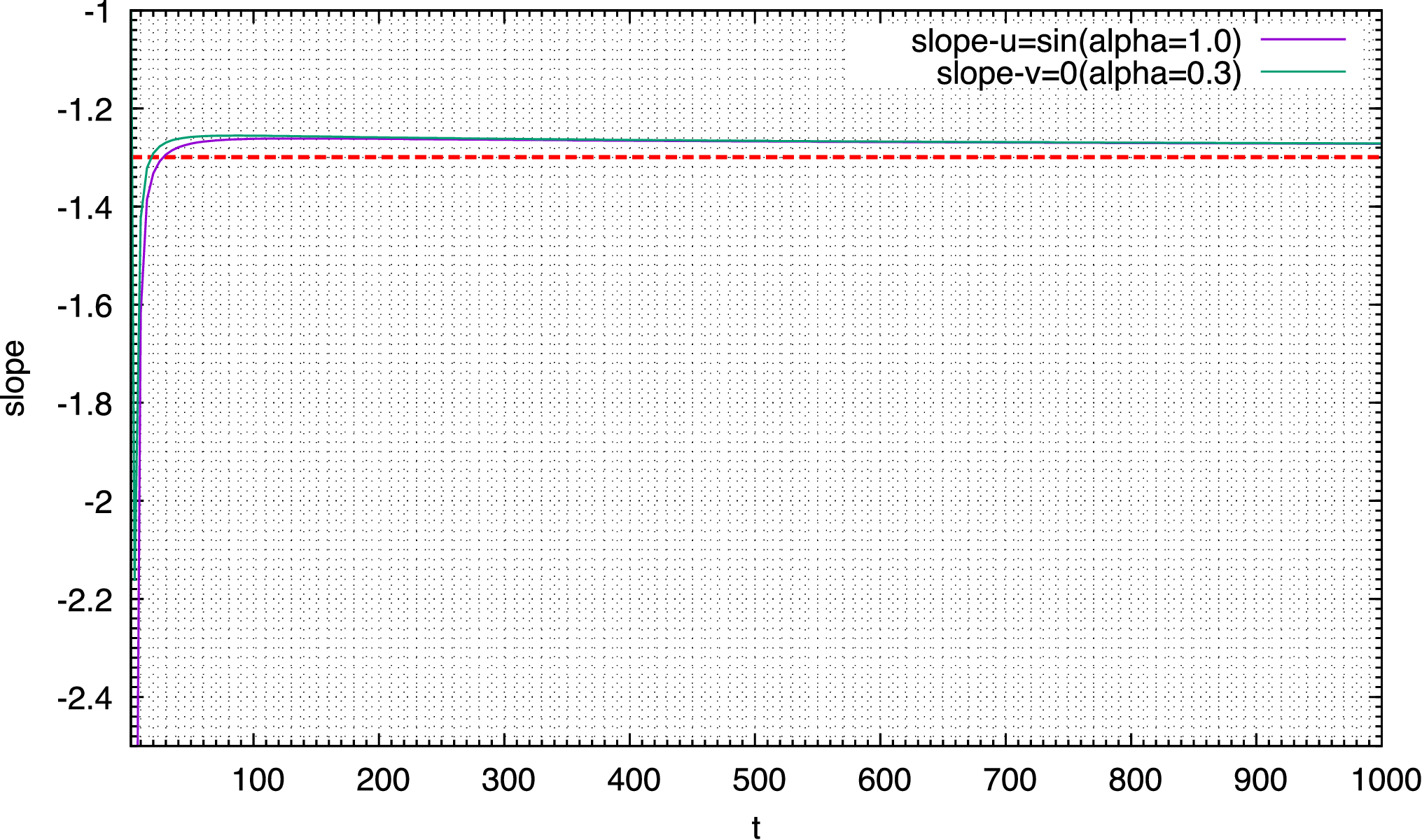}\quad
\includegraphics[width=0.48\textwidth]{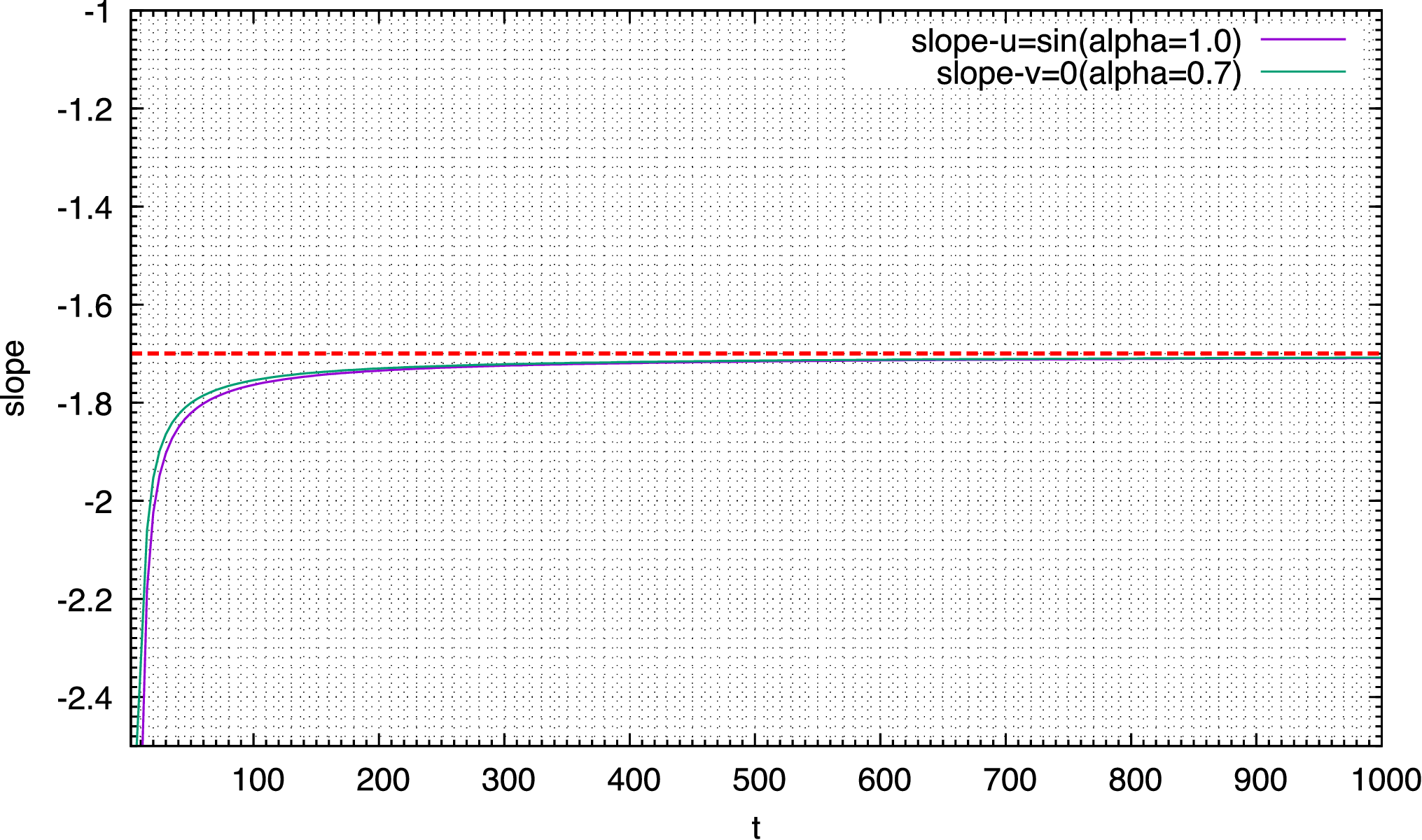}
\caption{Long-time asymptotic behavior of the solution to the coupled system \eqref{eq-gov2} with $\al=1.0$ and $v_0\equiv0$. Left: $\be=0.3$. Right: $\be=0.7$.}\label{asym-be_TPFDS_K=2}
\end{figure}
\end{example}

On the same direction, it is of natural curiosity to further study the long-time asymptotic behavior of coupling systems of 3 components at least from the numerical aspect. Inheriting the same formulation as \eqref{eq-gov2} of 2 components, here we deal with the model system
\begin{equation}\label{eq-gov3}
\begin{cases}
\begin{aligned}
& \pa_t^\al(u-u_0)-u_{xx}+u-0.5v-0.5w=0,\\
& \pa_t^\be(v-v_0)-v_{xx}-0.5u+v-0.5w=0,\\
& \pa_t^\ga(w-w_0)-w_{xx}-0.5u-0.5v+w=0
\end{aligned}
& \mbox{in }(0,\pi)\times(0,T),\\
u=v=w=0 & \mbox{on }\{0,\pi\}\times(0,T),
\end{cases}
\end{equation}
where $1\ge\al\ge\be\ge\ga>0$. The coupling coefficients $c_{k\ell}$ are safely chosen to achieve both the numerical stability and the possible assumption for the expected decay rate. For 3 components, the combinations of initial values are more flexible than before and here we consider the following 3 cases:
\begin{equation}\label{eq-IC3}
\left\{\!\begin{aligned}
\mbox{(i)}\ & u_0(x)=x(\pi-x),\ v_0(x)=\sin x,\ w_0(x)=\f\pi2-\left|x-\f\pi2\right|,\\
\mbox{(ii)}\ & u_0(x)=\sin x,\ v_0(x)=\f\pi2-\left|x-\f\pi2\right|,\ w_0(x)=0,\\
\mbox{(iii)}\ & u_0(x)=\sin x,\ v_0(x)=w_0(x)=0.
\end{aligned}\right.
\end{equation}
We perform similar simulations as before until the final time $T=1000$ and record the same quantities as \eqref{eq-numer-decay} to observe possible decay rates.

\begin{example}
Parallel to Example \ref{ex-2frac}, we start with the coupled system of 3 subdiffusion equations and choose the orders of time derivatives in \eqref{eq-gov3} as
\[
\al=0.9,\quad\be=0.5,\quad\ga=0.3.
\]
We test all 3 cases in \eqref{eq-IC3} for initial values and plot the time evolution of the decay for all 3 components of the solution in Figure \ref{asym-be_TFDS_K=3}. As before, we clearly observe that
\[
\left\{\begin{aligned}
\mbox{(i)}\ & \mbox{If $u_0\not\equiv0$, $v_0\not\equiv0$ and $w_0\not\equiv0$, then }\|u(t)\|,\|v(t)\|,\|w(t)\|\sim t^{-0.3},\\
\mbox{(ii)}\ & \mbox{If $u_0\not\equiv0$, $v_0\not\equiv0$ and $w_0\equiv0$, then }\|u(t)\|,\|v(t)\|,\|w(t)\|\sim t^{-0.5},\\
\mbox{(iii)}\ & \mbox{If $u_0\not\equiv0$, $v_0\equiv0$ and $w_0\equiv0$, then }\|u(t)\|,\|v(t)\|,\|w(t)\|\sim t^{-0.9}
\end{aligned}\right.\quad\mbox{as }t\to+\infty.
\]
Again, the result in Case (i) realizes the sharp decay rate $t^{-\ga}$ in Lemma \ref{lem-wp}(ii), i.e., the decay rate of all components depends on the smallest order of time derivatives as long as the initial value of the last component does not vanish identically. Meanwhile, the results in Cases (ii)--(iii) generalizes our observation in Example \ref{ex-2frac} and Theorem \ref{thm-asymp} in the sense that the decay is accelerated if the initial values of some components with smaller fractional orders vanish. More precisely, we can conjecture from the above numerical tests that the decay rate depends on the lowest fractional order whose initial value does not vanish identically. This corresponds with our intuitive explanation by the absence of initial supply, so that the lowest order with initial supply dominates the long-time behavior of the solution.
\begin{figure}[htbp]
\includegraphics[width=0.48\textwidth]{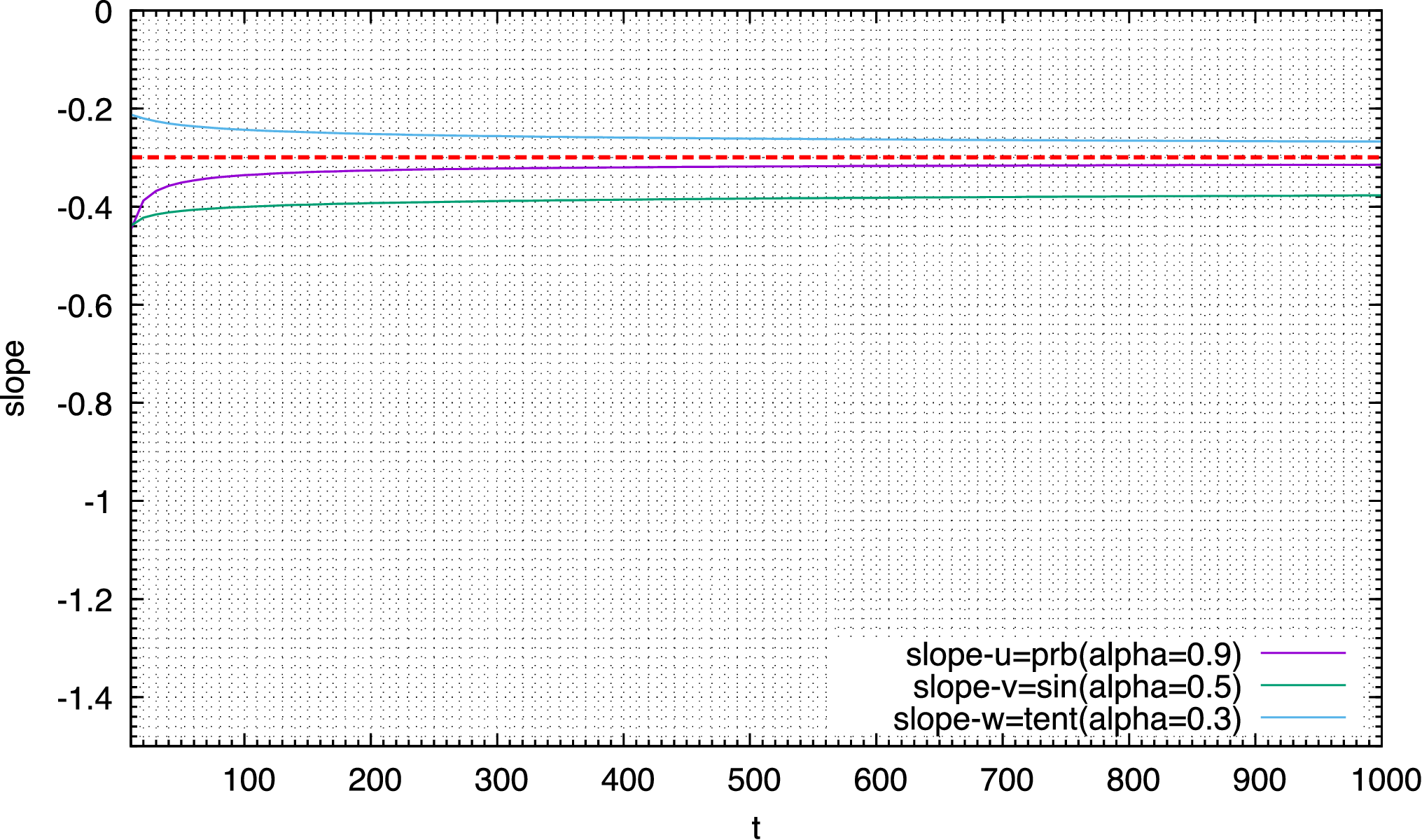}\quad
\includegraphics[width=0.48\textwidth]{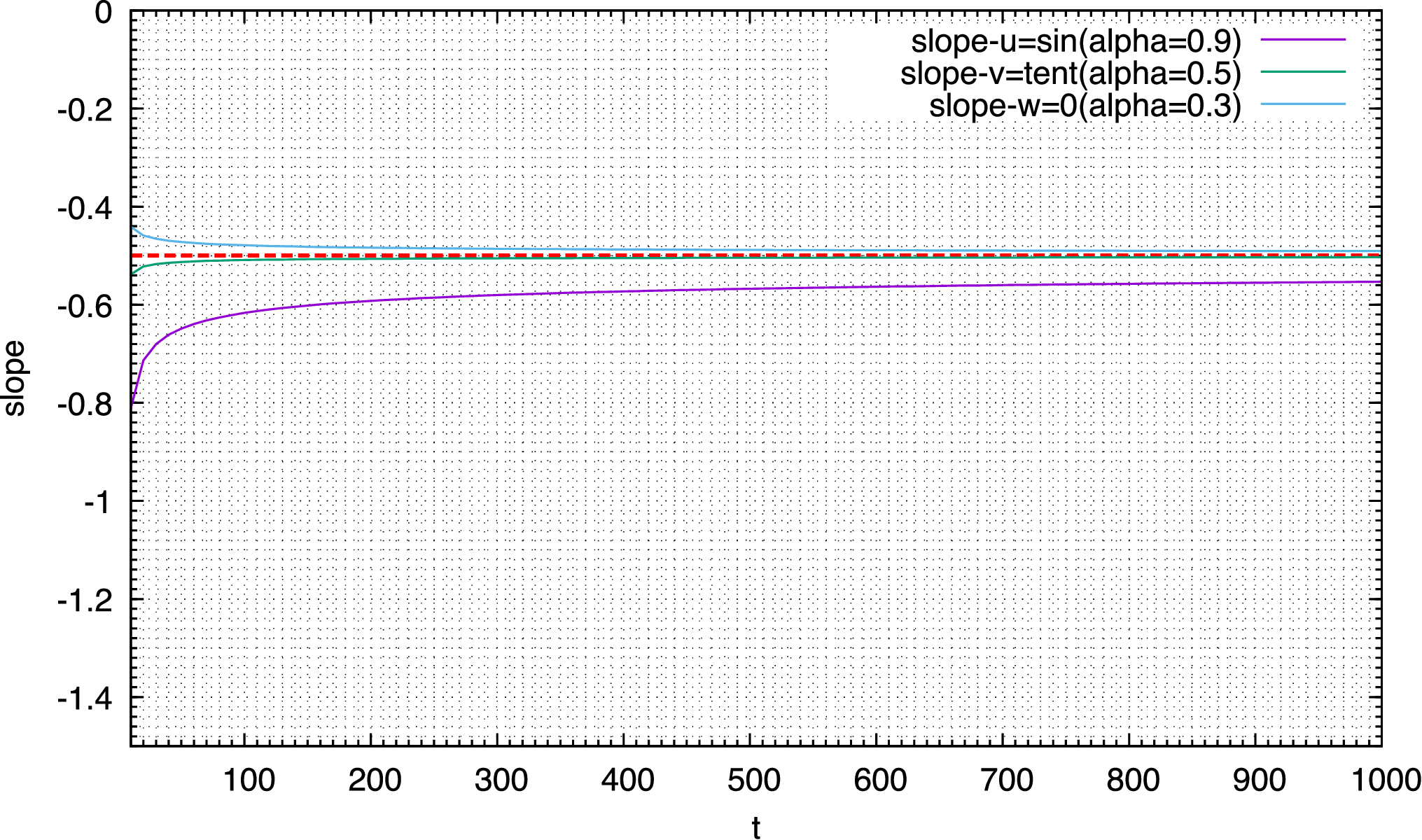}\\[3mm]
\includegraphics[width=0.48\textwidth]{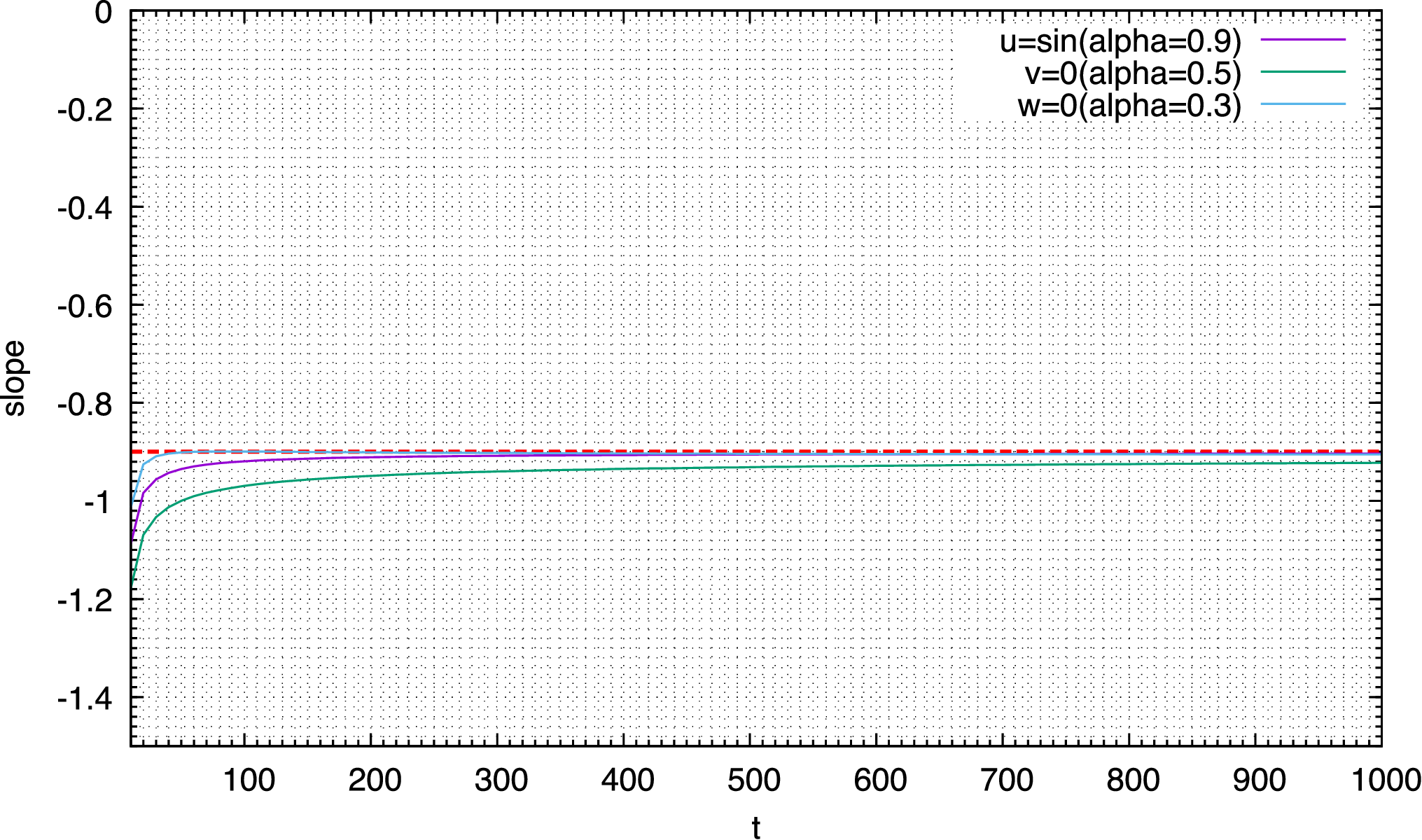}
\caption{Long-time asymptotic behavior of the solution to the coupled system \eqref{eq-gov3} with $\al=0.9$, $\be=0.5$ and $\ga=0.3$. Upper left: Case (i) of the initial values. Upper right: Case (ii) of the initial values. Bottom: Case (iii) of the initial values.}\label{asym-be_TFDS_K=3}
\end{figure}
\end{example}

Next, we investigate the decay pattern of the system with at least one usual diffusion equation and at least one subdiffusion equation, i.e., $\al=1$ and $\ga<1$ in \eqref{eq-gov3}. Motivated by the previous examples, it seems undoubtable that the decay rate of the solution to  \eqref{eq-gov3} should be $t^{-\ga}$ if $w_0\not\equiv0$. Therefore, we skip Case (i) in \eqref{eq-IC3} and mainly discuss Cases (ii)--(iii) in the following two examples.

\begin{example}
Consider Cases (ii) in \eqref{eq-IC3}, namely, $w_0\equiv0$ and $v_0\not\equiv0$. Since now we require $\al=1$ and $\ga<1$, we have some flexibility of choosing $\be$, i.e., $\be<1$ or $\be=1$. We implement both cases and demonstrate the decay of numerical solutions in Figure \ref{asym-be_TPFDS_K=3_w=0}, from which it follows that
\[
\left\{\begin{aligned}
& \mbox{If $\al=1.0$, $\be=0.5$ and $\ga=0.3$, then }\|u(t)\|,\|v(t)\|,\|w(t)\|\sim t^{-0.5},\\
& \mbox{If $\al=1.0$, $\be=1.0$ and $\ga=0.5$, then }\|u(t)\|,\|v(t)\|,\|w(t)\|\sim t^{-1.5}
\end{aligned}\right.\quad\mbox{as }t\to+\infty.
\]
Therefore, we observe similar decay patterns as those of 2 components in the sense that either sublinear or superlinear decay rate occurs depending on the value of $\be$.
\begin{figure}[htbp]
\includegraphics[width=0.48\textwidth]{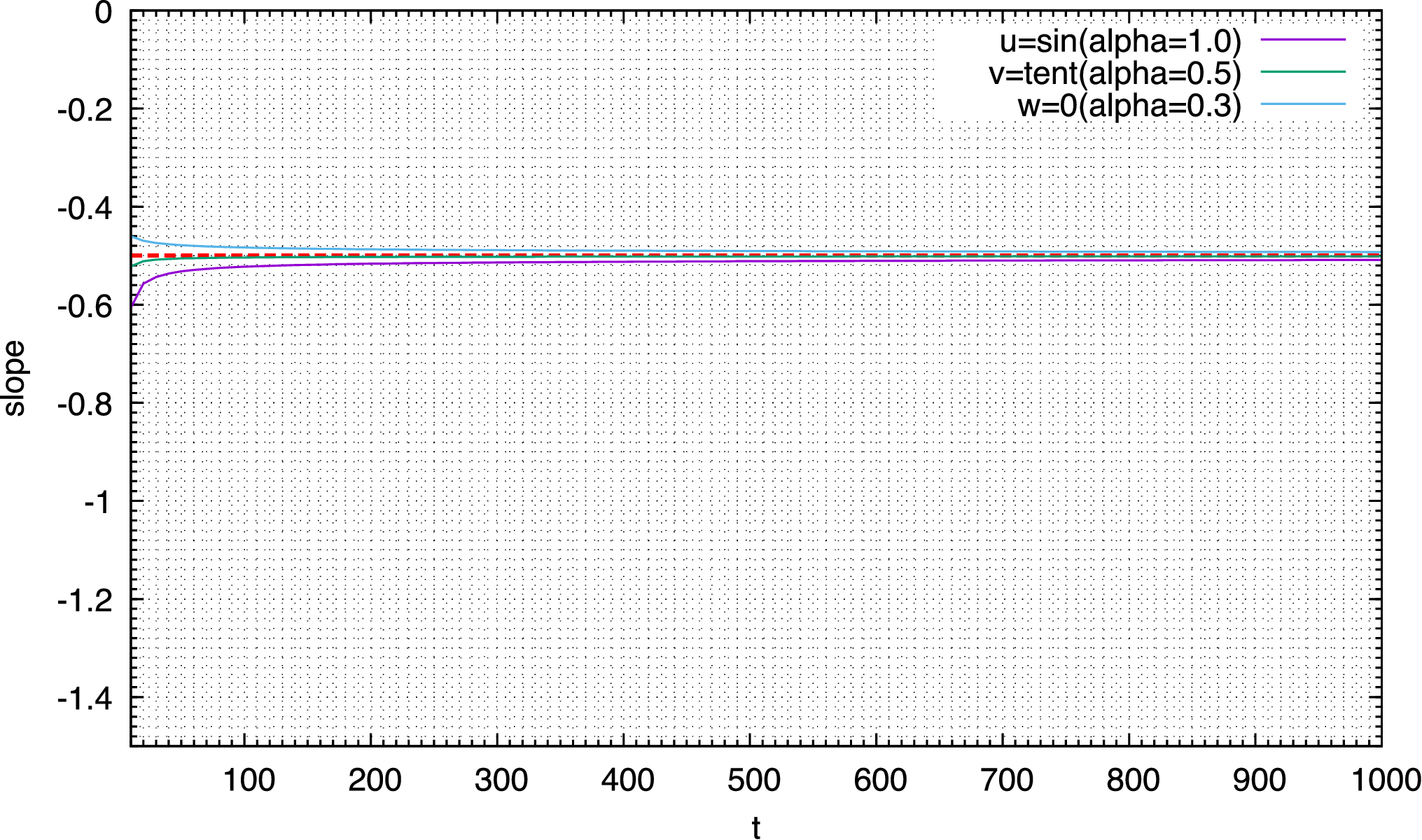}\quad
\includegraphics[width=0.48\textwidth]{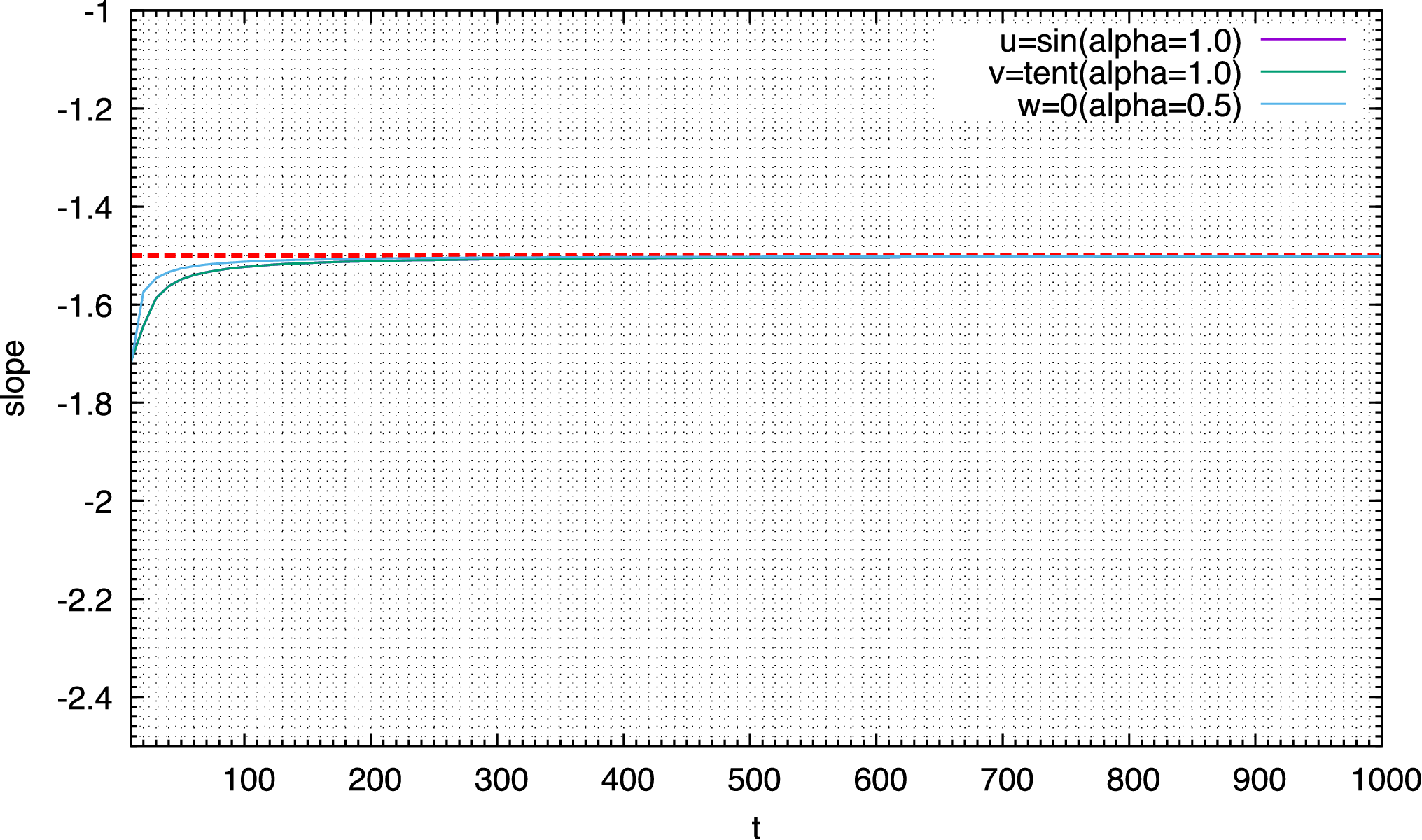}
\caption{Long-time asymptotic behavior of the solution to the coupled system \eqref{eq-gov3} with $w_0\equiv0$. Left: $\al=1.0$, $\be=0.5$ and $\ga=0.3$. Right: $\al=\be=1.0$ and $\ga=0.5$.}\label{asym-be_TPFDS_K=3_w=0}
\end{figure}

To further summarize the underlying rules of decay rates, we test various combinations of orders $\al,\be,\ga$. Since all figures until now show clear convergence to certain constants, we simply list the observed numerical results in Table \ref{asym-be_TPFDS_K=3_w=0_table}. Now it is obvious that the decay patterns switch according to the choice of $\be$, that is, $t^{-\be}$ (sublinear) if $\be<1$ and $t^{-(1+\ga)}$ (superlinear) if $\be=1$. Since $\be$ is the lowest order with a non-vanishing initial value, again we confirm the importance of this order as discussed in the last example.
\begin{table}[htbp]
\caption{Long-time asymptotic behavior of the solution to the coupled system \eqref{eq-gov3} with $w_0\equiv0$ and various choices of fractional orders.}\label{asym-be_TPFDS_K=3_w=0_table}
\begin{tabular}{ccc|c}
\hline\hline
$\al$ & $\be$ & $\ga$ & $\|u(t)\|,\|v(t)\|,\|w(t)\|$\\
\hline
$1.0$ & $0.5$ & $0.3$ & $t^{-0.5}$\\
$1.0$ & $0.5$ & $0.5$ & $t^{-0.5}$\\
$1.0$ & $0.7$ & $0.5$ & $t^{-0.7}$\\
$1.0$ & $1.0$ & $0.3$ & $t^{-1.3}$\\
$1.0$ & $1.0$ & $0.5$ & $t^{-1.5}$\\
$1.0$ & $1.0$ & $0.7$ & $t^{-1.7}$\\
\hline\hline
\end{tabular}
\end{table}
\end{example}

\begin{example}
Finally, Case (iii) in \eqref{eq-IC3} is considered, namely, $v_0=w_0\equiv0$ and $u_0\not\equiv0$. As before, we test both cases of $\be<1$ and $\be=1$ and illustrate the decay of numerical solutions in Figure \ref{asym-be_TPFDS_K=3_v=0-w=0}. Contrary to the last example, here both cases show a decay rate of $t^{-1.5}$ independent of the choice of $\be$.
\begin{figure}[htbp]
\includegraphics[width=0.48\textwidth]{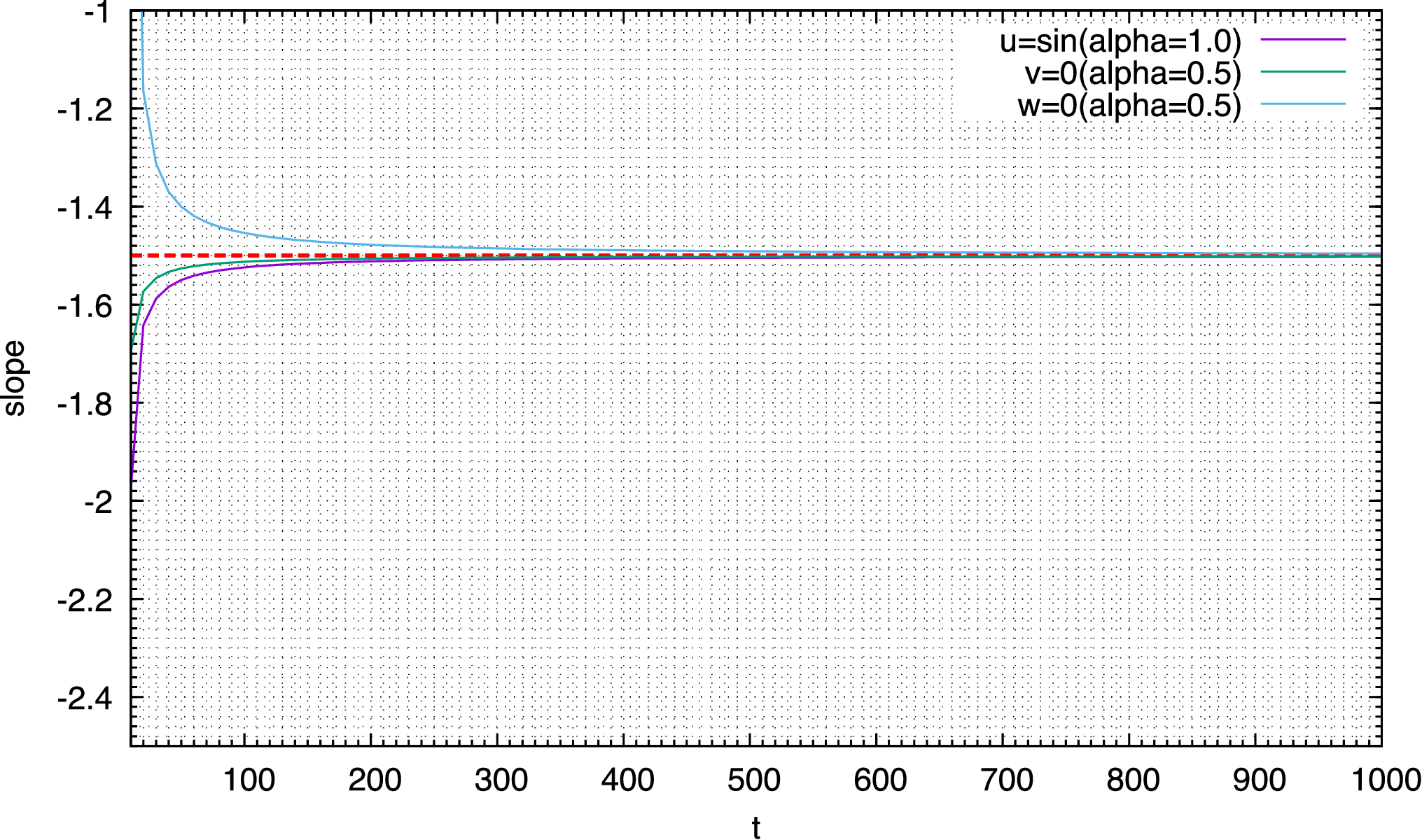}\quad
\includegraphics[width=0.48\textwidth]{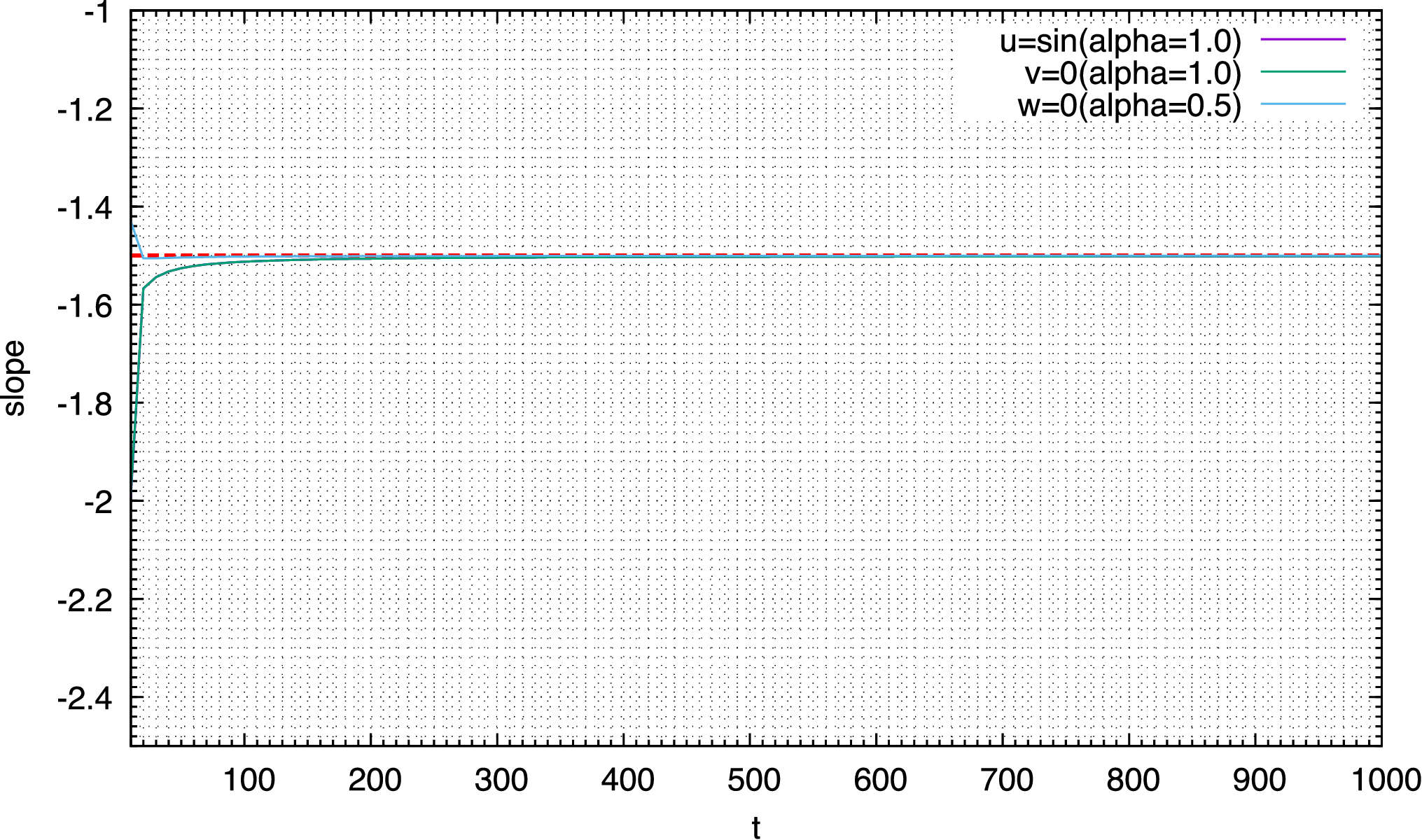}
\caption{Long-time asymptotic behavior of the solution to the coupled system \eqref{eq-gov3} with $v_0=w_0\equiv0$. Left: $\al=1.0$, $\be=0.5$ and $\ga=0.5$. Right: $\al=\be=1.0$ and $\ga=0.5$.}\label{asym-be_TPFDS_K=3_v=0-w=0}
\end{figure}

In order to better understand the mechanism, again we observe the decay patterns with various combinations of orders $\al,\be,\ga$. As listed in Table \ref{asym-be_TPFDS_K=3_v=0-w=0_table}, it reveals that all decay rates take the form of $t^{-(1+\ga)}$ regardless of the order $\be$ at the middle. In view of the initial supply, this time the lowest order with a non-vanishing initial value is exactly $1$. Then from Table \ref{asym-be_TPFDS_K=3_v=0-w=0_table}, we can conjecture that the decay rate in this case is always superlinear, whose power only depends on the lowest order ($\ga$ in this example) among all components.
\begin{table}[htbp]
\caption{Long-time asymptotic behavior of the solution to the coupled system \eqref{eq-gov3} with $v_0=w_0\equiv0$ and various choices of fractional orders.}\label{asym-be_TPFDS_K=3_v=0-w=0_table}
\begin{tabular}{ccc|c}
\hline\hline
$\al$ & $\be$ & $\ga$ & $\|u(t)\|,\|v(t)\|,\|w(t)\|$\\
\hline
$1.0$ & $0.5$ & $0.3$ & $t^{-1.3}$\\
$1.0$ & $0.5$ & $0.5$ & $t^{-1.5}$\\
$1.0$ & $0.7$ & $0.5$ & $t^{-1.5}$\\
$1.0$ & $1.0$ & $0.3$ & $t^{-1.3}$\\
$1.0$ & $1.0$ & $0.5$ & $t^{-1.5}$\\
$1.0$ & $1.0$ & $0.7$ & $t^{-1.7}$\\
\hline\hline
\end{tabular}
\end{table}
\end{example}

\section{Concluding remarks}

This research established a fundamental dichotomy in the long-time decay behavior of weakly coupled subdiffusion systems provided that the initial value of the lower-order component vanishes identically. When the highest fractional order satisfies $0<\alpha < 1$, the solutions exhibit characteristic sublinear decay $t^{-\alpha}$, consistent with the established pattern for the single fractional diffusion equation. Remarkably, systems with the highest order $\alpha = 1$ demonstrate a qualitatively distinct superlinear decay pattern $t^{-(1+\beta)}$. This accelerated decay phenomenon fundamentally distinguishes coupled systems from single equations, where such a superlinear decay pattern never occurs. 

As for the open problems related to the asymptotics for the coupled subdiffusion system, let us mention the decay patterns for
\begin{enumerate}
\item multi-component fractional differential systems, i.e., $K\ge3$, and
\item coupled time-fractional wave systems, i.e., the highest order is in $(1,2)$.
\end{enumerate}

For (i), in the previous section, we have collected sufficient hints from numerical experiments to propose a general conjecture on the sharp long-time decay rates for coupled subdiffusion systems with more than $2$ components.
\begin{conjecture*}
Let $K=2,3,\dots$ and consider the initial-boundary value problem for a coupled subdiffusion system with $K$ components
\begin{equation}\label{eq-couple1}
\left\{\begin{alignedat}{2}
& \pa_t^{\al_k}(u_k-u_0^{(k)})-\rdiv(\bm A_k(\bm x)\nb u_k)+\sum_{\ell=1}^K c_{k\ell}(\bm x,t)u_\ell=0 & \quad & \mbox{in }\Om\times\BR_+,\\
& u_k=0 & \quad & \mbox{on }\pa\Om\times\BR_+,
\end{alignedat}\right.
\end{equation}
$k=1,\dots,K,$ where $1\ge\al_1\ge\cdots\ge\al_K>0,$ $\al_K<1$ and the coefficients $\bm A_k,c_{k\ell}$ $(k,\ell=1,\dots,K)$ fulfill similar conditions as {\rm\eqref{eq-assume1}--\eqref{eq-assume2}}. Let $\un\al$ be the lowest order among $\al_1,\dots,\al_K$ whose corresponding initial value satisfies $u_0^{(k)}\not\equiv0$. Under some technical assumption such as $\eqref{eq-assume0},$ there exists a constant $C>0$ such that the solution to \eqref{eq-couple1} satisfy
\[
\|u_k(t)\|_{L^2(\Om)}\le\begin{cases}
C\,t^{-\un\al}, & \un\al<1,\\
C\,t^{-(1+\al_K)}, & \un\al=1,
\end{cases}\quad\forall\,t>1,\ k=1,\dots,K.
\]
\end{conjecture*}

In the case of $K=2$, the above conjecture reduces exactly to Theorem \ref{thm-asymp}. For $K=3$, it agrees with all numerical results obtained above and it seems highly possible to hold true for larger $K$. Nevertheless, the Laplace transform argument in the proof of Theorem \ref{thm-asymp} is unrealistically complicated for $3$ or more components. Hence, the verification of this conjecture awaits a different methodology, which is left as a future topic.

For (ii), the case $1<\alpha<2$ is indeed an important and interesting direction. In our current work, the decay pattern was established by employing the coercivity and maximum principle for fractional differential equations to derive our main results. However, for the case $1<\alpha<2$, the existing theoretical tools mentioned above are not yet well-established, and there are even known counterexamples. For instance, it was demonstrated in \cite{HL23} that solutions of fractional differential equations in this range may exhibit finite oscillation behavior and have finitely many zeros, which fundamentally differs from the case $0<\alpha<1$, see e.g., as shown in the paper \cite{LiuRunYa2016}.

We plan to further investigate the properties of solutions in these cases and explore potential extensions of mathematical tools in subsequent research.


\appendix
\section{Explicit decay rate of a decoupled subdiffusion system}

In Section \ref{sec-test}, the motivation of discovering Theorem \ref{thm-asymp} by means of numerical simulation was described. As a supplementary clue, in this appendix we provide another theoretical approach to finding the superlinear decay rate $t^{-(1+\be)}$ of the solution to a special decoupled system of a usual diffusion equation and a subdiffusion one.

Let us consider the initial-boundary value problem
\begin{equation}\label{eq-ibvp-nonsys1}
\begin{cases}
(\pa_t+\cA)u=0,\ (\pa_t^\be+\cA)v=u & \mbox{in }\Om\times\BR_+,\\
u=u_0\not\equiv0,\ v\equiv0 & \mbox{in }\Om\times\{0\},\\
u=v=0 & \mbox{on }\pa\Om\times\BR_+,
\end{cases}
\end{equation}
where $\be\in(0,1)$, $u_0\in L^2(\Om)$ and $\cA$ is a self-adjoint elliptic operator defined by
\[
\cA:H^2(\Om)\cap H_0^1(\Om)\longrightarrow L^2(\Om),\quad\psi\longmapsto-\rdiv(\bm A(\bx)\nb \psi)+c\psi.
\]
Here $\bm A=(a_{ij})_{1\le i,j\le d}\in C^1(\ov\Om;\BR_{\mathrm{sym}}^{d\times d})$ is the same matrix-valued function satisfying \eqref{eq-assume1}, and $0\le c\in L^\infty(\Om)$. Then it is well known that $\cA$ admits an eigensystem $\{(\la_n,\vp_n)\}_{n=1}^\infty$ such that
\[
0<\la_1<\la_2\le\cdots,\quad\la_n\longrightarrow\infty\ (n\to\infty),\quad\cA\vp_n=\la_n\vp_n\quad\mbox{in }\Om,
\]
and $\{\vp_n\}\subset H^2(\Om)\cap H_0^1(\Om)$ forms a complete orthonormal basis of $L^2(\Om)$.

It is readily seen that \eqref{eq-ibvp-nonsys1} is decoupled because $u$ does not depend on $v$, so that one can solve $u$ and $v$ one by one. The equation of $u$ is a usual parabolic one and of course $u(t)$ decays exponentially. On the other hand, we can calculate the explicit decay rate of $v(t)$ as follows.

\begin{lemma}\label{lem-sharp}
Let $(u,v)$ be the solution to \eqref{eq-ibvp-nonsys1}. Then there holds
\[
\left\|v(t)-\f{u_\infty}{-\Ga(-\be)}t^{-(1+\be)}\right\|_{H^6(\Om)}=o(t^{-(1+\be)}),\quad t\gg1,
\]
where $u_\infty=\cA^{-3}u_0$ is the solution to the boundary value problem for a triple elliptic equation
\[
\begin{cases}
\cA^3u_\infty=u_0 & \mbox{in }\Om,\\
u_\infty=\cA u_\infty=\cA^2u_\infty=0 & \mbox{on }\pa\Om.
\end{cases}
\]
\end{lemma}

\begin{proof}
Employing the eigensystem $\{(\la_n,\vp_n)\}$ of $\cA$ and the Mittag-Leffler function, we can easily represent the explicit solution to \eqref{eq-ibvp-nonsys1} as
\[
u(t)=\sum_{n=1}^\infty\e^{-\la_n t}(u_0,\vp_n)\vp_n,\quad v(t)=\sum_{n=1}^\infty\int_0^t\tau^{\be-1}E_{\be,\be}(-\la_n\tau^\be)(u(t-\tau),\vp_n)\vp_n\,\rd\tau.
\]
Then we plug the expression of $u$ into that of $v$ to write
\begin{align*}
v(t) & =\sum_{n=1}^\infty\int_0^t\tau^{\be-1}E_{\be,\be}(-\la_n\tau^\be)\left(\sum_{n=1}^\infty\e^{-\la_n(t-\tau)}(u_0,\vp_n)\vp_n,\vp_n\right)\vp_n\,\rd\tau\\
& =\sum_{n=1}^\infty\int_0^t\tau^{\be-1}E_{\be,\be}(-\la_n\tau^\be)\,\e^{-\la_n(t-\tau)}\,\rd\tau\,(u_0,\vp_n)\vp_n=\sum_{n=1}^\infty P(t)(u_0,\vp_n)\vp_n,
\end{align*}
where
\[
P(t):=\int_0^t\tau^{\be-1}\sum_{k=0}^\infty\f{(-\la_n\tau^\be)^k}{\Ga(\be k+\be)}\sum_{k=0}^\infty\f{(-\la_n(t-\tau))^k}{\Ga(k+1)}\,\rd\tau.
\]
Next, we deal with $P(t)$ as
\begin{align*}
P(t) & =\int_0^t\tau^{\be-1}\sum_{k=0}^\infty\sum_{j=0}^\infty\f{(-\la_n\tau^\be)^j{(-\la_n(t-\tau)})^{k-j}}{\Ga(\be j+\be)\Ga(k-j+1)}\,\rd\tau\\
& =\int_0^t\tau^{\be-1}\sum_{k=0}^\infty(-\la_n)^k\sum_{j=0}^k\f{\tau^{\be j}(t-\tau)^{k-j}}{\Ga(\be j+\be)\Ga(k-j+1)}\,\rd\tau\\
& =\sum_{k=0}^\infty(-\la_n)^k\sum_{j=0}^k\int_0^t\f{\tau^{\be(j+1)-1}(t-\tau)^{k-j}}{\Ga(\be j+\be)\Ga(k-j+1)}\,\rd\tau=\sum_{k=0}^\infty(-\la_n)^k\sum_{j=0}^k Q(t, j),
\end{align*}
where
\[
Q(t,j):=\int_0^t\f{\tau^{\be(j+1)-1}(t-\tau)^{k-j}}{\Ga(\be j+\be)\Ga(k-j+1)}\,\rd\tau.
\]
Perform the integration by substitution by $\tau=\te t$, $\te\in(0,1)$, we calculate $Q(t,j)$ as
\begin{align*}
Q(t,j) & =\int_0^t\f{\tau^{\be(j+1)-1}(t-\tau)^{k-j}}{\Ga(\be j+\be)\Ga(k-j+1)}\,\rd\tau=\int_0^1\f{(\te t)^{\be(j+1)-1} ((1-\te)t)^{k-j} }{\Ga(\be j+\be)\Ga(k-j+1)}t\,\rd\te\\
& =t^{\be(j+1)-1+k-j+1}\int_0^1\f{\te^{\be(j+1)-1}(1-\te)^{k-j}}{\Ga(\be j+\be)\Ga(k-j+1)}\,\rd\te\\
& =t^{\be(j+1)+k-j}\f{B(\be j+\be,k-j+1)}{\Ga(\be j+\be)\Ga(k-j+1)}=\f{t^{\be(j+1)+k-j}}{\Ga(\be j+k-j+\be+1)},
\end{align*}
where $B(\,\cdot\,, \,\cdot\,)$ denotes the Beta function. Then substituting $Q(t,j)$ back into $P(t)$ implies
\begin{align*}
P(t) & =\sum_{k=0}^\infty(-\la_n)^k\sum_{j=0}^k Q(t, j)=\sum_{k=0}^\infty(-\la_n)^k\sum_{j=0}^k\f{t^{\be(j+1)+k-j}}{\Ga(\be j+k-j+\be+1)}\\
& =t^\be\sum_{k=0}^\infty(-\la_n)^k\sum_{j=0}^k\f{t^{\be j+(k-j)}}{\Ga(\be j+(k-j)+\be+1)}=:t^\be R(t).
\end{align*}

Now we concentrate on the series $R(t)$. Rearranging the terms in $R(t)$ according to the power of $-\la_n t^\be$, we recall the definition of the Mittag-Leffler functions to calculate
\begin{align*}
R(t) & =\sum_{k=0}^\infty(-\la_n)^k\sum_{j=0}^k\f{t^{\be j+(k-j)}}{\Ga(\be j+(k-j)+\be+1)}\\
& =\f1{\Ga(\be+1)}+(-\la_n)\left\{\f t{\Ga(\be+2)}+\f{t^\be}{\Ga(2\be+1)}\right\}\\
& \quad\,+(-\la_n)^2\left\{\f{t^2}{\Ga(\be+3)}+\f{t^{\be+1}}{\Ga(2\be+2)}+\f{t^{2\be}}{\Ga(3\be+1)}\right\}\\
& \quad\,+(-\la_n)^3\left\{\f{t^3}{\Ga(\be+4)}+\f{t^{\be+2}}{\Ga(2\be+3)}+\f{t^{2\be+1}}{\Ga(3\be+2)}+\f{t^{3\be}}{\Ga(4\be+1)}\right\}+\cdots\\
& =\left\{\f1{\Ga(\be+1)}+\f{(-\la_n t)}{\Ga(\be+2)}+\f{(-\la_n t)^2}{\Ga(\be+3)}+\f{(-\la_n t)^3}{\Ga(\be+4)}+\cdots\right\}\\
& \quad\,+(-\la_n t^\be)\left\{\f1{\Ga(2\be+1)}+\f{(-\la_n t)}{\Ga(2\be+2)}+\f{(-\la_n t)^2}{\Ga(2\be+3)}+\cdots\right\}\\
& \quad\,+(-\la_n t^\be)^2\left\{\f1{\Ga(3\be+1)}+\f{(-\la_n t)}{\Ga(3\be+2)}+\cdots\right\}+\cdots\\
& =E_{1,\be+1}(-\la_n t)+(-\la_n t^\be) E_{1,2\be+1}(-\la_n t)+(-\la_n t^\be)^2E_{1,3\be+1}(-\la_n t)+\cdots\\
& =\sum_{k=0}^\infty(-\la_n t^\be)^k E_{1,\be(k+1)+1}(-\la_n t) .
\end{align*}
Then plugging the above expression of $R(t)$ back into $P(t)$ and then into $v(t)$ yields
\begin{align*}
v(t) & =\sum_{n=1}^\infty P(t)(u_0,\vp_n)\vp_n=\sum_{n=1}^\infty t^\be R(t)(u_0,\vp_n)\vp_n\\
& =t^\be\sum_{n=1}^\infty\sum_{k=0}^\infty(-\la_n t^\be)^k E_{1,\be(k+1)+1}(-\la_n t) (u_0,\vp_n)\vp_n.
\end{align*}
Now we invoke the following asymptotic behavior for the Mittag-Leffler function $E_{\eta,\mu}(-z)$ with $\eta\in(0,2)$, $\mu>0$ and $z<0$ (see Podlubny \cite[Theorem 1.4]{PF98}):
\begin{equation}\label{eq-asymp-ML}
E_{\eta,\mu}(z)=-\f{z^{-1}}{\Ga(\mu-\eta)}-\f{z^{-2}}{\Ga(\mu-2\eta)}+O(|z|^{-3})\quad\mbox{as }z\to-\infty.
\end{equation}
Then for $t\gg1$, we can take advantage of \eqref{eq-asymp-ML} to approximate $v(t)$ as
\begin{align*}
v(t) & =t^\be\sum_{n=1}^\infty\sum_{k=0}^\infty(-\la_n t^\be)^k\left\{\f1{\Ga(\be(k+1))\la_n t}-\f1{\Ga(\be(k+1)-1)\la_n^2t^2}+O\left(\f1{(\la_n t)^3}\right)\right\}(u_0,\vp_n)\vp_n\\
& =t^{\be-1}\sum_{n=1}^\infty E_{\be,\be}(-\la_n t^\be)\f{(u_0,\vp_n)}{\la_n}\vp_n-t^{\be-2}\sum_{n=1}^\infty E_{\be,\be-1}(-\la_n t^\be)\f{(u_0,\vp_n)}{\la_n^2}\vp_n+O(t^{\be-3}).
\end{align*}
Applying the asymptotic estimate \eqref{eq-asymp-ML} again to the Mittag-Leffler functions above yields
\begin{align}
v(t) & =t^{\be-1}\sum_{n=1}^\infty\f1{-\Ga(-\be)(\la_n t^\be)^2}\f{(u_0,\vp_n)}{\la_n}\vp_n+t^{\be-2}\sum_{n=1}^\infty\f1{-\Ga(-1-\be)(\la_n t^\be)^2}\f{(u_0,\vp_n)}{\la_n^2}\vp_n+O(t^{\be-3})\nonumber\\
& =\f{t^{-(1+\be)}}{-\Ga(-\be)}\sum_{n=1}^\infty\f{(u_0,\vp_n)}{\la_n^3}\vp_n+\f{t^{-(2+\be)}}{-\Ga(-1-\be)}\sum_{n=1}^\infty\f{(u_0,\vp_n)}{\la_n^4}\vp_n+O(t^{\be-3})\nonumber\\
& =\f{\cA^{-3}u_0}{-\Ga(-\be)}t^{-(1+\be)}+\f{\cA^{-4}u_0}{\Ga(-1-\be)}t^{-(2+\be)}+O(t^{\be-3}),\quad t\gg1.\label{eq-decay}
\end{align}
Here we interpret $\f1{\Ga(0)}=\f1{\Ga(-1)}=0$ in the sense of limit, and we notice that
\[
-\Ga(-\be)>0,\quad O(t^{\be-3})=o(t^{-(1+\be)}),\ t\gg1
\]
by $\be\in(0,1)$. Therefore, the last 2 terms on the right-hand side of \eqref{eq-decay} are of order $o(t^{-(1+\be)})$ and are as smooth as $\cA^{-4}u_0$. Consequently, recalling $u_\infty=\cA^{-3}u_0$, we take $H^6(\Om)$ norm on both sides of \eqref{eq-decay} to conclude
\begin{align*}
\left\|v(t)-\f{u_\infty}{-\Ga(-\be)}t^{-(1+\be)}\right\|_{H^6(\Om)} & \le C\left\|\cA^3v(t)-\f{u_0}{-\Ga(-\be)}t^{-(1+\be)}\right\|_{L^2(\Om)}\\
& =C\|\cA^{-1}u_0\|_{L^2(\Om)}o(t^{-(1+\be)})=o(t^{-(1+\be)}),\quad t\gg1,
\end{align*}
where $C>0$ is a constant depending only on $\cA$. The proof is complete.
\end{proof}

Owing to the simplicity of \eqref{eq-ibvp-nonsys1}, we can easily write down its explicit solution, so that the above lemma provides not only the sharp decay rate of $v$ but also its limit pattern $u_\infty$, i.e., the profile of $v$ approaches $u_\infty$ as $t\to+\infty$. In other words, $v$ asymptotically takes the form of separated variables for large $t$ with the spatial component $\f{u_\infty}{-\Ga(-\be)}$ and the temporal component $t^{-(1+\be)}$. 

Moreover, based on the decay estimates, we demonstrate that if the solution decays faster than $t^{-(1+\beta)}$, then its initial value $u_0$ must be zero.
\begin{corollary}
Let $(u,v)$ be the solution to the fractional diffusion system \eqref{eq-ibvp-nonsys1} with initial data $u(x,0)=u_0(x)$ and $v(x,0)=0$. If there exists a constant $C > 0$ such that
$$
\|v(t)\|_{H^6(\Omega)} \leq Ct^{-(1+\beta)} \eta(t)
$$
with an infinitesimal $\eta(t)$ as $t\to\infty$. Then the initial value $u_0$ must be zero.
\end{corollary}
The proof is trivial and is omitted. This corollary implies that any non-trivial initial data cannot yield a solution that decays at a rate exceeding the critical rate in Lemma \ref{lem-sharp}.  A similar result for the single fractional diffusion-wave equation can be found in e.g., Yamamoto \cite{Yama2024}.

The limit pattern in more general situations is complicated, which can be another interesting future topic.

\section*{Acknowledgments}

The authors thank the anonymous referees for valuable comments. The first author thanks the National Natural Science Foundation of China (12271277), and Ningbo Youth Leading Talent Project (2024QL045).  The second author is supported by JSPS KAKENHI Grant Numbers JP22K13954, JP23KK0049 and Guangdong Basic and Applied Basic Research Foundation (No.\! 2025A1515012248).



\begin{thebibliography}{00}

\bibitem{AS72} 
Abramowitz, M., Stegun, I.A.: Handbook of Mathematical Functions with Formulas, Graphs, and Mathematical Tables. Dover Publications, New York (1964).

\bibitem{Adam1992}
Adams, E.E., Gelhar, L.W.: Field study of dispersion in a heterogeneous aquifer 2. Spatial moments analysis. Water Resour. Res. {\bf28}(12), 3293--3307 (1992).

\bibitem{Doerries2022}
Doerries, T.J., Chechkin, A.V., Schumer, R., et al.: Rate equations, spatial moments, and concentration profiles for mobile-immobile models with power-law and mixed waiting time distributions. Phys. Rev. E {\bf105}(1), 014105 (2022). 

\bibitem{GL15}
Gorenflo, R., Luchko, Y., Yamamoto, M.: Time-fractional diffusion equation in the fractional Sobolev spaces. Fract. Calc. Appl. Anal. {\bf18}(3), 799--820 (2015).

\bibitem{Hatano1998}
Hatano, Y., Hatano, N.: Dispersive transport of ions in column experiments: an explanation of long-tailed profiles. Water Resour. Res. {\bf34}(5), 1027--1033 (1998).

\bibitem{HL23}
Huang, X., Liu, Y.: Long-time asymptotic estimate and a related inverse source problem for time-fractional wave equations. In: Takiguchi, T., et al. (eds.), Practical Inverse Problems and Their Prospects, pp. 163--179, Springer, Singapore (2023).

\bibitem{Kubica}
Kubica, A., Ryszewska, K.: Decay of solutions to parabolic-type problem with distributed order Caputo derivative. J. Math. Anal. Appl. {\bf465}(1), 75--99 (2018).

\bibitem{KR20}
Kubica, A., Ryszewska, K., Yamamoto, M.: Time-Fractional Differential Equations: A Theoretical Introduction. Springer, Singapore (2020).

\bibitem{LHL23}
Li, Z., Huang, X., Liu, Y.: Initial-boundary value problems for coupled systems of time-fractional diffusion equations. Fract. Calc. Appl. Anal. {\bf26}(2), 533--566 (2023).

\bibitem{LHY21}
Li, Z., Huang, X., Yamamoto, M.: Well-posedness and asymptotic estimate for a diffusion equation with time-fractional derivative. Chin. Ann. Math. Ser. B {\bf46}(1), 115--138 (2025).

\bibitem{LLY15}
Li, Z., Liu, Y.,Yamamoto, M.: Initial-boundary value problems for multi-term time-fractional diffusion equations with positive constant coefficients. Appl. Math. Comput. {\bf257}, 381--397 (2015).

\bibitem{LiWen2020}
Li, X., Wen, Z., Zhu, Q., et al.: A mobile-immobile model for reactive solute transport in a radial two-zone confined aquifer. J. Hydrol. {\bf580}, 124347 (2020). 

\bibitem{LiuRunYa2016}
Liu, Y., Rundell, W., Yamamoto, M.: Strong maximum principle for fractional diffusion equations and an application to an inverse source problem. Fract. Calc. Appl. Anal., {\bf 19}, 888--906 (2016).

\bibitem{Luchko2014}
Li, Z., Luchko, Y., Yamamoto, M.: Asymptotic estimates of solutions to initial-boundary-value problems for distributed order time-fractional diffusion equations. Fract. Calc. Appl. Anal. {\bf17}(4), 1114--1136 (2014).

\bibitem{LX07}
Lin, Y., Xu, C.: Finite difference/spectral approximations for the time-fractional diffusion equation. J. Comput. Phys. {\bf225}, 1533--1552 (2007).

\bibitem{Metzler2000}
Metzler, R., Klafter, J.: The random walk's guide to anomalous diffusion: a fractional dynamics approach. Phys. Rep. {\bf339}(1), 1--77 (2000).

\bibitem{PF98}
Podlubny, I.: Fractional Differential Equations. Academic Press, San Diego (1999).

\bibitem{SY11}
Sakamoto, K., Yamamoto, M.: Initial value/boundary value problems for fractional diffusion-wave equations and applications to some inverse problems. J. Math. Anal. Appl. {\bf382}(1), 426--447 (2011).

\bibitem{Schumer2003}
Schumer, R., Benson, D.A., et al.: Fractal mobile/immobile solute transport. Water Resour. Res. {\bf39}(10), 1296--1308  (2003).

\bibitem{SunNiu2021}
Sun, L., Niu, J., Huang, F., et al.: An efficient fractional-in-time transient storage model for simulating the multi-peaked breakthrough curves. J. Hydrol. {\bf600}, 126570 (2021).

\bibitem{Zacher2015}
Vergara, V., Zacher, R.: Optimal decay estimates for time-fractional and other nonlocal subdiffusion equations via energy methods. SIAM J. Math. Anal. {\bf47}(1), 210--239 (2015).

\bibitem{Yama2024}
Yamamoto, M: Decay rates and initial values for time-fractional diffusion-wave equations. Ann. Acad. Rom. Sci. Ser. Math. Appl. {\bf16}(1), 77-97 (2024).

\end{thebibliography}
\end{document}